\documentclass[11pt]{article}
\def\thetitle{A generalised Ramsey--Tur\'an problem for matchings}

\usepackage{graphicx}

\usepackage{amsmath,amssymb}
\usepackage{nicefrac}

\usepackage[usenames,dvipsnames,svgnames,table]{xcolor}
\definecolor{CombinatoricaAqua}{HTML}{00698C}
\definecolor{CombinatoricaBlue}{HTML}{3A3293}
\definecolor{CombinatoricaBrown}{HTML}{66220C}
\definecolor{CombinatoricaRed}{HTML}{DF2A27}
\definecolor{HarvardCrimson}{rgb}{0.6471, 0.1098, 0.1882}
\definecolor{DAGreen}{HTML}{339900}

\makeatletter
\let\reftagform@=\tagform@
\def\tagform@#1{\maketag@@@
	{(\ignorespaces\textcolor{CombinatoricaBrown}{#1}\unskip\@@italiccorr)}}
\renewcommand{\eqref}[1]{\textup{\reftagform@{\ref{#1}}}}
\makeatother

\usepackage[backref=page]{hyperref}
\hypersetup{%
	unicode,
	pdfencoding=auto,
	pdfauthor={Peleg Michaeli},
	pdftitle={\thetitle},
	pdfsubject={},
	pdfkeywords={},
	colorlinks=true,
	citecolor=CombinatoricaBlue,
	linkcolor=CombinatoricaAqua,
	anchorcolor=CombinatoricaBrown,
	urlcolor=HarvardCrimson}

\usepackage{amsthm}
\usepackage{thmtools}
\usepackage{bbm}
\usepackage{enumitem}
\usepackage[capitalize]{cleveref}
\Crefname{mainthm}{Theorem}{Theorems}
\Crefname{fact}{Fact}{Facts}
\Crefname{claim}{Claim}{Claims}
\Crefname{observation}{Observation}{Observations}
\Crefname{definition}{Definition}{Definitions}

\declaretheoremstyle[
spaceabove=\topsep, spacebelow=\topsep,
headfont=\color{CombinatoricaBrown}\normalfont\bfseries,
bodyfont=\itshape,
]{thm}
\declaretheoremstyle[
spaceabove=\topsep, spacebelow=\topsep,
headfont=\color{CombinatoricaBrown}\normalfont\bfseries,
bodyfont=\normalfont,
]{dfn}
\declaretheoremstyle[
spaceabove=0.5\topsep, spacebelow=0.5\topsep,
headfont=\color{CombinatoricaBrown}\normalfont\bfseries,
bodyfont=\normalfont,
]{rmk}

\declaretheorem[style=thm,parent=section]{theorem}
\declaretheorem[style=thm,sibling=theorem]{lemma}
\declaretheorem[style=thm,sibling=theorem]{corollary}
\declaretheorem[style=thm,sibling=theorem]{claim}

\declaretheorem[style=definition,sibling=theorem]{definition}

\usepackage[nobysame,msc-links,non-sorted-cites]{amsrefs}

\renewcommand{\eprint}[1]{\href{https://arxiv.org/abs/#1}{arXiv:#1}}

\BibSpec{book}{%
	+{}  {\PrintPrimary}                {transition}
	+{,} { \textbf}                     {title} % was \textit
	+{.} { }                            {part}
	+{:} { \textit}                     {subtitle}
	+{,} { \PrintEdition}               {edition}
	+{}  { \PrintEditorsB}              {editor}
	+{,} { \PrintTranslatorsC}          {translator}
	+{,} { \PrintContributions}         {contribution}
	+{,} { }                            {series}
	+{,} { \voltext}                    {volume}
	+{,} { }                            {publisher}
	+{,} { }                            {organization}
	+{,} { }                            {address}
	+{,} { \PrintDateB}                 {date}
	+{,} { }                            {status}
	+{}  { \parenthesize}               {language}
	+{}  { \PrintTranslation}           {translation}
	+{;} { \PrintReprint}               {reprint}
	+{.} { }                            {note}
	+{.} {}                             {transition}
	+{}  {\SentenceSpace \PrintReviews} {review}
}
\BibSpec{incollection}{%
  +{}  {\PrintAuthors}                {author}
  +{,} { \textit}                     {title}
  +{.} { }                            {part}
  +{:} { \textit}                     {subtitle}
  +{,} { \PrintContributions}         {contribution}
  +{,} { \PrintConference}            {conference}
  +{}  {\PrintBook}                   {book}
  +{,} { }                            {booktitle}
	+{}  { \PrintEditorsB}              {editor}
	+{,} { }                            {publisher}
  +{,} { \PrintDateB}                 {date}
  +{,} { pp.~}                        {pages}
  +{,} { }                            {status}
  +{,} { \PrintDOI}                   {doi}
  +{,} { available at \eprint}        {eprint}
  +{}  { \parenthesize}               {language}
  +{}  { \PrintTranslation}           {translation}
  +{;} { \PrintReprint}               {reprint}
  +{.} { }                            {note}
  +{.} {}                             {transition}
  +{}  {\SentenceSpace \PrintReviews} {review}
}

\makeatletter
\renewcommand{\PrintNames@a}[4]{%
	\PrintSeries{\name}
	{#1}
	{}{ and \set@othername}
	{,}{ \set@othername}
	{}{ and \set@othername}
	{#2}{#4}{#3}%
}
\makeatother

\makeatletter
\def\mathcolor#1#{\@mathcolor{#1}}
\def\@mathcolor#1#2#3{%
	\protect\leavevmode
	\begingroup
	\color#1{#2}#3%
	\endgroup
}
\makeatother
\definecolor{Red}{rgb}{0.618,0,0}
\definecolor{Blue}{rgb}{0,0,1}
\definecolor{Green}{rgb}{0,0.298,0}

\usepackage{sectsty}
\sectionfont{\color{CombinatoricaBrown}}
\subsectionfont{\color{CombinatoricaBrown}}
\subsubsectionfont{\color{CombinatoricaBrown}}

\usepackage{algorithm}
\usepackage[noend]{algpseudocode}

\usepackage{soul}
\soulregister\ref7
\usepackage{pifont}
\usepackage{calc}

\usepackage[a4paper]{geometry}
\title{\thetitle}
\author{
  Peter Keevash\thanks{
    Mathematical Institute,
    University of Oxford,
    Oxford, UK.
    Email: \href{mailto:keevash@maths.ox.ac.uk}
                {\tt keevash@maths.ox.ac.uk}.
    Supported by ERC Advanced Grant 883810.
  }
  \and
  Peleg Michaeli\thanks{
    Mathematical Institute,
    University of Oxford,
    Oxford, UK.
    Email: \href{mailto:michaeli@maths.ox.ac.uk}
                {\tt michaeli@maths.ox.ac.uk}.
    Supported by ERC Advanced Grant 883810.
  }
}

\makeatletter
\def\namedlabel#1#2{\begingroup
  #2%
  \def\@currentlabel{#2}%
  \phantomsection\label{#1}\endgroup
}
\makeatother

\usepackage{mleftright}
\mleftright
\newcommand{\defn}[1]{{\bfseries #1}}

\renewcommand{\phi}{\varphi}
\newcommand{\NN}{\mathbb{N}}
\newcommand{\ZZ}{\mathbb{Z}}

\newcommand{\cA}{\mathcal{A}}

\newcommand{\cG}{\mathcal{G}}
\newcommand{\cH}{\mathcal{H}}

\newcommand{\cK}{\mathcal{K}}

\newcommand{\cP}{\mathcal{P}}

\newcommand{\cX}{\mathcal{X}}
\newcommand{\cY}{\mathcal{Y}}
\newcommand{\sm}{\smallsetminus}
\newcommand{\es}{\varnothing}

\newcommand{\floor}[1]{\left\lfloor{#1}\right\rfloor}

\newcommand{\vect}{\mathbf}

\DeclareMathOperator{\ex}{ex}
\newcommand{\rt}{\mathbf{RT}}
\newcommand{\grt}{\mathbf{GRT}}
\newcommand{\nto}{\not\to}

\DeclareMathOperator{\link}{link}
\DeclareMathOperator{\unclr}{uncolour}
\newcommand{\vt}{\mathbf{t}}
\newcommand{\tmax}{\|\vt\|_\infty}
\newcommand{\one}{\mathbf{1}}

\usepackage{tabto}
\usepackage{tcolorbox}
\usepackage{comment}

\usepackage{subcaption}
\usepackage{tikz,pgfplots}
\pgfplotsset{compat=1.16}
\usetikzlibrary{calc}
\usetikzlibrary{fit}

\newcommand{\GE}{\mathsf{GE}}

\newcommand{\mns}{\nu_\Sigma} % Matching number sum

\begin{document}

\maketitle

\begin{abstract}
We prove a generalised Ramsey--Tur\'an theorem for matchings, which 
(a) simultaneously generalises the Cockayne--Lorimer Theorem (Ramsey for matchings)
and the Erd\H{o}s--Gallai Theorem (Tur\'an for matchings),
and (b) is a generalised Tur\'an theorem in the sense that we can optimise
the count of any clique (Tur\'an-type theorems optimise the count of edges).
More precisely, for integers $q \ge 1$, $n \ge \ell \ge 2$, and $t_1,\dots,t_q \ge 1$
we determine the maximum number of $\ell$-vertex complete subgraphs
in an $n$-vertex graph that admits a $q$-edge-colouring in which,
for each $j=1,\dots,q$, the $j$-coloured subgraph has no matching of size $t_j$.
We achieve this by identifying two explicit constructions 
and applying a compression argument to show that 
one of them achieves the maximum. Our compression algorithm
is quite intricate and introduces methods that have not previously 
been applied to these types of problems:
it employs an optimisation problem
defined by the Gallai--Edmonds decompositions of each colour.
\end{abstract}

\section{Introduction}
Extremal Combinatorics is a pillar of Discrete Mathematics with applications to many other fields, 
particularly Theoretical Computer Science, Geometry and Number Theory,
see~\cite{sudakov2010recent}.
Many results of the field can be classified as Ramsey or Tur\'an Theorems, 
which have a rich history dating to the early 20th century,
see~\cites{conlon2015recent,mubayi2020survey,keevash2011hypergraph,mubayi2016survey}.
A hybrid of these two research directions proposed by Erd\H{o}s and S\'os in the 1960s
is now known as Ramsey--Tur\'an Theory, see~\cite{SS01}. Another generalisation of Tur\'an problems,
known simply as Generalised Tur\'an Problems~\cite{GP25+}, developed from a generalisation
of Tur\'an's Theorem proved in the 1960s by  Erd\H{o}s~\cite{erdos1962number}.
Our paper concerns a common generalisation
of Ramsey--Tur\'an Theory and Generalised Tur\'an Problems,
which we will now define.

Given graphs $T$ and $G$, we write $m_T(G)$ for 
the number of unlabelled copies of $T$ in $G$.
Given graphs $H_1,\dots,H_q$, we write $G\to (H_1,\dots,H_q)$
if every $q$-edge-colouring of $G$ has a $j$-coloured copy of $H_j$ for some $j\in[q]:=\{1,\dots,q\}$.
We define the \defn{Generalised Ramsey--Tur\'an number} by
\[
  \grt_T(n\to(H_1,\dots,H_q))
  = \max\{m_T(G)\ :\ G\nto(H_1,\dots,H_q),\ |V(G)|=n\}.
\]
We also write $m_\ell=m_{K_\ell}$, $\grt_\ell=\grt_{K_\ell}$
and $tK_2$ for a matching of size $t$.

Our first result is structural, showing that
$\grt_\ell(n\to(t_1K_2,\dots,t_qK_2))$ is achieved by a graph
from a certain concrete family, which we will now define.
Let $G_{n,x,y}$ be an $n$-vertex graph on $X\cup Y\cup Z$,
where $X,Y,Z$ are pairwise disjoint, $|X|=x$ and $|Y|=y$,
and in which $\{u,v\}$ is an edge if and only if
either $u,v\in X$ or $\{u,v\}\cap Y\ne\es$.
In other words, $G_{n,x,y}$ is the join of $K_y$
with the disjoint union of $K_x$ and $n-x-y$ isolated vertices.
We call such $G$ a \defn{clique-cone graph},
with \defn{clique set} $X$ and  \defn{cone set} $Y$.
Write $\NN$ for the set of nonnegative integers
and $\NN_+ = \NN \sm \{0\}$.
For $\vt=(t_1,\dots,t_q)\in\NN_+^q$,
we use $\vt K_2$ as an abbreviation for $(t_1K_2,\dots,t_q K_2)$.
\begin{theorem}\label{thm:max}
  For any integers $n\ge\ell\ge 2$, $q \ge 1$ and $\vt\in\NN_+^q$, 
  there exist $1\le x\le 2\tmax-1$ and $0\le y\le n-x$
  such that $G_{n,x,y}\nto\vt K_2$
  and $m_\ell(G_{n,x,y})=\grt_\ell(n\to\vt K_2)$.
\end{theorem}

Our second theorem gives an explicit formula for $\grt_\ell(n\to\vt K_2)$. We write
\begin{equation}\label{eq:phi}
  m_\ell(G_{n,x,y})=\phi_{\ell,n}(x,y)=\binom{x+y}{\ell}+\binom{y}{\ell-1}(n-x-y).
\end{equation}
We also write $\one_q=(1,\dots,1)\in\NN_+^q$
and $\Lambda_\vt=\|\vt-\one_q\|_1 = \sum_{i=1}^q (t_i-1)$ for $\vt\in\NN_+^q$.
\begin{theorem}\label{thm:grt}
  For all $\vt\in\NN_+^q$ and $n\ge\max\{\ell,\tmax+\Lambda_\vt\}$ we have
  \[
    \grt_\ell(n\to\vt K_2)
    = \max\{
        \phi_{\ell,n}(1,\Lambda_\vt),
        \phi_{\ell,n}(2\tmax-1,\Lambda_\vt-\tmax+1)
      \}.
  \]
  Also, $\grt_\ell(n\to \vt K_2)=0$ for $n<\ell$
  and $\grt_\ell(n\to \vt K_2)=\binom{n}{\ell}$ for $\ell \le n \le \tmax+\Lambda_\vt$.
\end{theorem}

As discussed below, the $\ell=2$ case of this theorem simultaneously generalises
the Erd\H{o}s--Gallai Theorem (Tur\'an for matchings)~\cite{EG61}
and the Cockayne--Lorimer Theorem (Ramsey for matchings)~\cite{CL75}.
In the non-trivial case $n\ge\max\{\ell,\tmax+\Lambda_\vt\}$,
it identifies the extremal graph for $\grt_\ell(n\to\vt K_2)$
as $G_{n,1,\Lambda_\vt}$ or $G_{n,2\tmax-1,\Lambda_\vt-\tmax+1}$,
whichever yields a larger count of $K_\ell$'s.

For these to be candidate graphs $G$ for $\grt_\ell(n\to\vt K_2)$,
we must also provide colourings
demonstrating $G \nto\vt K_2$, which we do below, after some further definitions.
A \defn{$q$-multi-colouring} of a graph $G=(V,E)$
is a sequence $\cG=(G_1,\dots,G_q)$
where $G_j=(V,E_j)$ is a simple graph for all $j\in[q]$,
and $\bigcup_j E_j=E$.
We call $G$ the \defn{underlying graph} 
of the \defn{coloured graph} $\cG$.
An edge $e$ is said to have colour $j$ if $e\in E_j$.
Note that we allow an edge to possess multiple colours
and that in the definition of $\grt$ it is equivalent 
to consider multi-colourings or usual colourings,
where each edge has exactly one colour.
We will therefore abuse terminology and also
refer to $q$-multi-colourings as \defn{$q$-colourings}.
We say that $\cG$ is $(H_1,\dots,H_q)$-free if $G_j$ is $H_j$-free for all $j\in[q]$.
The \defn{matching number} $\nu(G)$ of $G$
is the size of a maximum matching in $G$.
Write $\nu(\cG)=(\nu(G_1),\dots,\nu(G_q))$,
and observe that $G\nto\vt K_2$
if and only if there exists a $q$-colouring $\cG$ of $G$
which is $\vt K_2$-free, i.e.~$\nu(\cG)\le\vt-\one_q$,
with inequalities between vectors understood pointwise.

\begin{figure}[t]
\captionsetup{width=0.879\textwidth,font=small}
    \centering
    \begin{subfigure}[b]{0.45\textwidth}
        \centering
        \includegraphics[width=\textwidth]{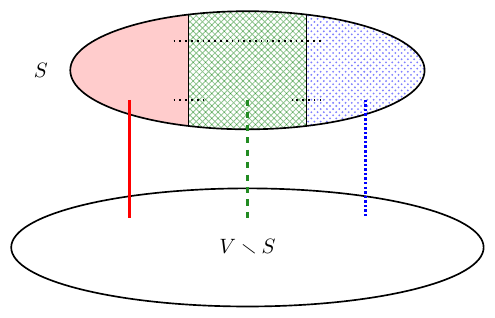}
        \subcaption{Sparse construction.}
        \label{fig:csparse}
    \end{subfigure}
    \hfill
    \begin{subfigure}[b]{0.45\textwidth}
        \centering
        \includegraphics[width=\textwidth]{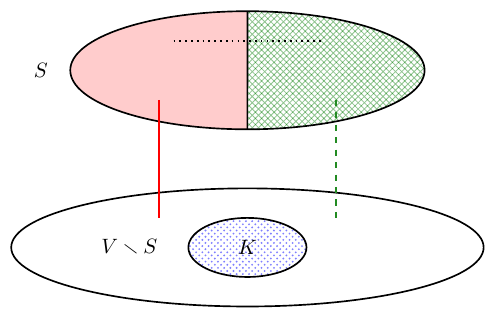}
        \subcaption{Dense construction.}
        \label{fig:cdense}
    \end{subfigure}
    \caption{Extremal constructions ($q=3$).
      The thin dotted lines correspond to multi-coloured edges.}
    \label{fig:const}
\end{figure}

\paragraph{Sparse construction}
Let $\vt=(t_1,\dots,t_q)\in\NN_+^q$
and $n\ge\max\{\ell,\tmax+\Lambda_\vt\}$.
Consider the following $q$-colouring $\cG=(G_1,\dots,G_q)$
of $G = G_{n,1,\Lambda_\vt}$. Partition the cone set $Y$
into sets $Y_i$ of size $|Y_i|=t_i-1$ for $1 \le i \le q$.
We let each $G_i$ consist of all edges incident to $Y_i$; see \cref{fig:csparse}.
Then $m_\ell(G)=\phi_{\ell,n}(1,\Lambda_\vt)$
and $\nu(\cG)\le\vt-\one_q$, so $G \nto \vt K_2$.

\paragraph{Dense construction}
Let $\vt=(t_1,\dots,t_q)\in\NN_+^q$
and $n\ge\max\{\ell,\tmax+\Lambda_\vt\}$.
Consider the following $q$-colouring $\cG=(G_1,\dots,G_q)$
of $G = G_{n,2\tmax-1,\Lambda_\vt-\tmax+1}$.
Let $X$ be the clique set and $Y$ be the cone set.
We assume without loss of generality that $t_q=\tmax$
and partition $Y$ into sets $Y_j$ of size $|Y_j|=t_j-1$ for $1 \le j \le q-1$.
For $1 \le j \le q-1$ we let $G_j$ consist of all edges incident to $Y_j$,
and we let $G_q$ consist of all edges contained in $X$; see \cref{fig:cdense}.
Then $m_\ell(G)=\phi_{\ell,n}(2\tmax-1,\Lambda_\vt-\tmax+1)$
and $\nu(\cG)\le\vt-\one_q$, so $G \nto \vt K_2$.

~

Note that the above two constructions do not depend on $\ell$,
although which of them achieves $\grt_\ell(n\to\vt K_2)$ does depend on $\ell$.
When there is a single colour ($q=1$) we obtain the two extremal constructions
for the Erd\H{o}s--Gallai Theorem~\cite{EG61}.
When $n=\tmax+\Lambda_\vt$ we obtain a $\vt K_2$-free $q$-colouring of $K_n$
that is extremal for the Cockayne--Lorimer Theorem~\cite{CL75};
this construction establishes the last statement in \Cref{thm:grt}.

\subsection{Discussion and related results}
Let us place our result within the wider context of 
Tur\'an, Ramsey, and Ramsey--Tur\'an problems for matchings.
The \defn{Tur\'an number} $\ex(n,H)$
is the maximum number of edges in an $H$-free graph on $n$ vertices.
Thus $\ex(n,H)=\grt_2(n\to (H))$. 
The Erd\H{o}s--Gallai Theorem~\cite{EG61} states that
\begin{equation}\label{eq:h:T}
  \grt_2(n\to (tK_2)) =\ex(n,tK_2)
  = \max\left\{
    \phi_{2,n}(1,t-1),
    \phi_{2,n}(2t-1,0)
  \right\}.
\end{equation}
This is the case  $q=1$, $\ell=2$, $\vt=(t)$ of
our formula in \cref{thm:grt}.

The \defn{generalised Tur\'an number} $\ex(n,T,H)$ 
is the maximum number of unlabelled copies of $T$
in an $H$-free graph on $n$ vertices.
Thus $\ex(n,T,H)=\grt_T(n\to (H))$. 
When $T$ and $H$ are cliques, this number
was determined by Erd\H{o}s~\cite{erdos1962number},
thus generalising Tur\'an's Theorem.
The general study of $\ex(n,T,H)$ was initiated by
Alon and Shikhelman~\cite{AS16},
and it now has a substantial literature, surveyed in~\cite{GP25+}.
For matchings, Wang et al.~\cite{Wan20,DNPWY20}
generalised \eqref{eq:h:T} to
\begin{equation}\label{eq:h:GT}
   \grt_\ell(n\to (tK_2)) = \ex(n,K_\ell,tK_2)
  = \max\left\{
    \phi_{\ell,n}(2t-1,0),
    \phi_{\ell,n}(1,t-1)
  \right\}.
\end{equation}
Gerbner~\cite{Ger24+} recently obtained further results
replacing $T=K_\ell$ by more general graphs $T$.

The \defn{Ramsey number} of $H_1,\dots,H_q$, denoted $R(H_1,\dots,H_q)$,
is the smallest integer $n$ such that for every $q$-edge-colouring of $K_n$
there is a copy of $H_j$ in colour $j$ for some $j \in [q]$.
Thus $R(H_1,\dots,H_q)$ is the smallest integer $n$
such that $\grt_T(n\to(H_1,\dots,H_q))<m_T(K_n)$,
for any nonempty graph $T$.
For matchings, where $H_j=t_jK_2$ for $\vt=(t_1,\dots,t_q)\in\NN_+^q$,
the Cockayne--Lorimer Theorem~\cite{CL75} determines
$R(\vt K_2) = \tmax+\Lambda_\vt+1$ for all $\vt$.
Equivalently, $n_0=\tmax+\Lambda_\vt+1$
is the smallest integer for which $\grt_2(n_0 \to \vt K_2) < \binom{n_0}{2}$,
which follows from the $\ell=2$ case of \cref{thm:grt};
indeed, for $n<n_0$ we have $\grt_2(n\to \vt K_2) = \binom{n}{2}$,
whereas $\grt_2(n\to \vt K_2)$ is achieved by
a non-complete graph for $n \ge n_0$.

The \defn{Ramsey--Tur\'an number} $\rt(n;H_1,\dots,H_q,\alpha)$
is the maximum number of edges in an $n$-vertex graph $G$ with 
$G\not\to(H_1,\dots,H_q)$ and independence number $\alpha(G) \le \alpha$.
Thus $\grt_2(n\to(H_1,\dots,H_q))=\rt(n;H_1,\dots,H_q,n+1)$;
one could consider a further generalisation of $\grt$ with a nontrivial
constraint on $\alpha$, but we will not pursue this here.
Starting with Erd\H{o}s and S\'os in the 1960s, this problem
was studied in many papers of Erd\H{o}s et al., surveyed in~\cite{SS01},
and more recently in~\cites{BL13,BHS15}. For matchings,
Omidi and Raeisi~\cite{OR24} showed that
\begin{equation}\label{eq:h:RT}
  \grt_2(n\to \vt K_2) = \max\left\{
      \phi_{2,n}(1,\Lambda_\vt),
      \phi_{2,n}(2\tmax-1,\Lambda_\vt-\tmax+1)
  \right\}.
\end{equation}

As one might expect given the above definitions, 
in {\em generalised Ramsey--Turán theory},
introduced by  Balogh, Liu, and Sharifzadeh~\cite{BLS17},
the objective shifts from counting edges to counting copies of some given graph.
Our result, \cref{thm:grt}, provides such a generalisation of  \cref{eq:h:GT,eq:h:RT},
showing that the maximum is determined by one of the same two extremal graphs as in the $\ell=2$ case,
although which is the maximiser depends on $\ell$.

We emphasise that our generalisation to $\ell > 2$ requires a fundamentally new approach,
as the proof technique used in~\cite{OR24} for $\ell=2$ does not extend.
The main difficulty arises from the trivial fact that in coloured graphs
all copies of $K_2$ are monochromatic but this fails for $K_\ell$ with $\ell>2$,
so one cannot count $K_\ell$'s by considering each colour separately.
The proof in~\cite{OR24} uses the Cockayne--Lorimer Theorem as a black box,
whereas our proof contains a proof for Cockayne--Lorimer as a special case
(a similar proof appears in~\cite{XYZ20}). The key to our proof is a novel compression algorithm
that is quite intricate and introduces methods that have not previously 
been applied to these types of problems:
it employs an optimisation problem
defined by the Gallai--Edmonds decompositions of each colour.

\subsection{Proof outline}
Here we outline our strategy for proving our theorems.
Given an arbitrary $\vt K_2$-free coloured graph,
we apply a sequence of {\em compressions},
each producing a new coloured graph, so that
(a) the count of $K_\ell$'s in the underlying graph does not decrease,
and (b) the matching number for each individual colour does not increase.
Thus throughout the process we maintain a $\vt K_2$-free coloured graph
with no fewer $K_\ell$'s than the original. We find a sequence of such compressions
that terminates with a clique-cone graph, thus showing that  an extremal graph 
can always be found within this specific family, as claimed by \cref{thm:max}.
There are three stages to the main compression algorithm.
\begin{enumerate}
\item
Given a coloured graph, we first simplify the structure of each colour class individually
by adding edges in a controlled manner determined by its
Gallai--Edmonds decomposition (see \cref{sec:GE}),
without changing its matching number (see \cref{alg:CAD:comp,alg:CA:trans,alg:AD:comp}).
\item
Next we select a certain subset $T$ of vertices, where for each $x \in T$ 
we remove all coloured edges incident to $x$ 
and add {\em uncoloured} edges from $x$ to all other vertices.
We choose $T$ to solve a certain optimisation problem
defined by the Gallai--Edmonds decompositions of the coloured graphs.
This can be intuitively understood as removing sets that are too dense;
after uncolouring edges at $T$, the hypergraph of cliques in the remaining coloured graphs
has a forest-like structure that can be exploited in the third stage.
Furthermore, the sum of matching numbers over all colours
decreases by at least $|T|$, which will later allow us to recolour
the uncoloured edges while still ensuring that the colouring remains $\vt K_2$-free.
\item
We iteratively simplify the forest-like structure of cliques
by ``peeling'' it from its leaves (see \cref{alg:peel}).
At each iteration, we can remove a leaf clique
by one of the following two methods.
If there exists a clique of a different colour which is no smaller,
then we can {\em dissolve} the leaf (see \cref{alg:dissolve}) 
by adding about half of its vertices to $T$
and isolating its remaining vertices
(which are chosen not in any other clique).
If no such clique exists, then we can {\em merge} the leaf
(see \cref{alg:D:merge}) into another clique of the same colour.
This peeling procedure preserves the forest-like structure,
allowing the process to continue iteratively
until we are left with at most one nontrivial clique.
This final clique is the clique set of the resulting clique-cone graph,
while the set of uncoloured vertices forms its cone set.
\end{enumerate}
The above algorithm shows that some clique-cone graph is extremal,
thus proving  \cref{thm:max}. To deduce \cref{thm:grt}, we consider
a certain necessary condition, which we call {\em admissibility},
on the sizes $x$ and $y$ of the clique set and cone set 
in a clique-cone graph that admits a $\vt K_2$-free colouring.
We then show that $\phi_{\ell,n}(x,y)$ is convex along the relevant boundary
of the admissible region, which implies that its maximum is achieved at one of two points, 
which correspond to the sparse and dense constructions described above, thus completing the proof.

\section{Single-colour compressions}\label{sec:manip}
This section describes various compressions 
that will be applied to the individual coloured graphs.
These are defined using the Gallai--Edmonds decompositions,
which are described in \cref{sec:GE}.
In \cref{sec:completion} we describe the {\em completion} 
compressions in Stage~1 of the main algorithm.
\Cref{sec:isolation,sec:merging} describe
the clique isolation and clique merging algorithms
used for peeling leaves in Stage~3 of the main algorithm, see \Cref{sec:peel}.

\subsection{Gallai--Edmonds decompositions}\label{sec:GE}
We start by defining the \defn{Gallai--Edmonds decomposition}
(for short, \defn{GE-decomposition})  $\GE(G)=(C,A,D)$
of a graph $G=(V,E)$. We call a vertex \defn{essential} in $G$
if it is covered by every maximum matching of $G$,
or otherwise \defn{inessential}.
We let $D\subseteq V$ be the set of inessential vertices,
let $A$ be the set of vertices in $V\sm D$ 
adjacent to at least one vertex of $D$,
and let $C=V\sm (D\cup A)$.

The utility of the GE-decomposition is demonstrated by the
following Gallai--Edmonds Structure theorem;
see~\cite{LP}*{Theorem~3.2.2}.
For $U\subseteq V$, we write
$k(U)=k_G(U)$ for the number of connected components 
in the induced subgraph $G[U]$ on $U$.
We say that a matching of a graph $H$ is \defn{near-perfect}
if it leaves exactly one vertex uncovered (so $|V(H)|$ is odd).
We say that $H$ is \defn{factor-critical} if $H \sm v$
has a perfect matching for every $v \in V(H)$.
\begin{theorem}[Gallai--Edmonds Structure Theorem]
  \label{thm:GE}
  Let $G=(V,E)$ be a graph with $\GE(G)=(C,A,D)$. 
  Then each component of $D$ is factor-critical.
  Also, any maximum matching of $G$ contains a perfect matching of $C$
  and a near-perfect matching of each component of $D$,
  and matches all vertices of $A$ with vertices in distinct components of $D$.
  In particular, $|V|-2\nu(G)=k(D)-|A|$ vertices are uncovered.
  \end{theorem}

We also require the following {\em stability lemma} (see~\cite{LP}*{Lemma~3.2.2}).
\begin{lemma}[Stability]\label{lem:stab}
  Let $G=(V,E)$ be a graph with $\GE(G)=(C,A,D)$.
  Let $v \in A$ and $G'=G \sm v = G[V \sm \{v\}]$.
  Then $\GE(G')=(C,A\sm\{v\},D)$.
\end{lemma}

For a set of edges $E$, we write $\tau(E)$
for the minimum size $|T|$ of a set $T$ of vertices
that is a \defn{cover} for $E$, meaning that
$T\cap e\ne\es$ for every $e\in E$.
For future reference, we note the following condition
for every vertex of a cover to be essential.

\begin{lemma}[Covers]\label{lem:cover}
  Let $G=(V,E)$ be a graph and let $F$ be a set of non-edges of $G$.
  Obtain $G^+$ from $G$ by adding $F$ as edges.
  Then $\nu(G^+)\le \nu(G)+\tau(F)$.
  Moreover, if $\nu(G^+)=\nu(G)+\tau(F)$
  and $T$ is a cover for $F$ with $|T|=\tau(F)$,
  then every maximum matching of $G^+$ covers $T$.
\end{lemma}

\begin{proof}
  Let $T$ be a minimum cover for $F$
  and $M^+$ be a maximum matching in $G^+$.
  Then $M^+$ has at most $\nu(G)$ edges of $E$
  and at most $\tau(F)$ edges of $F$,
  so $\nu(G^+)=|M^+|\le\nu(G)+\tau(F)$.
  If $M^+$ misses a vertex from a minimum cover for $F$
  then the same argument gives $|M^+|<\nu(G)+\tau(F)$.
\end{proof}

\subsection{Completion}\label{sec:completion}

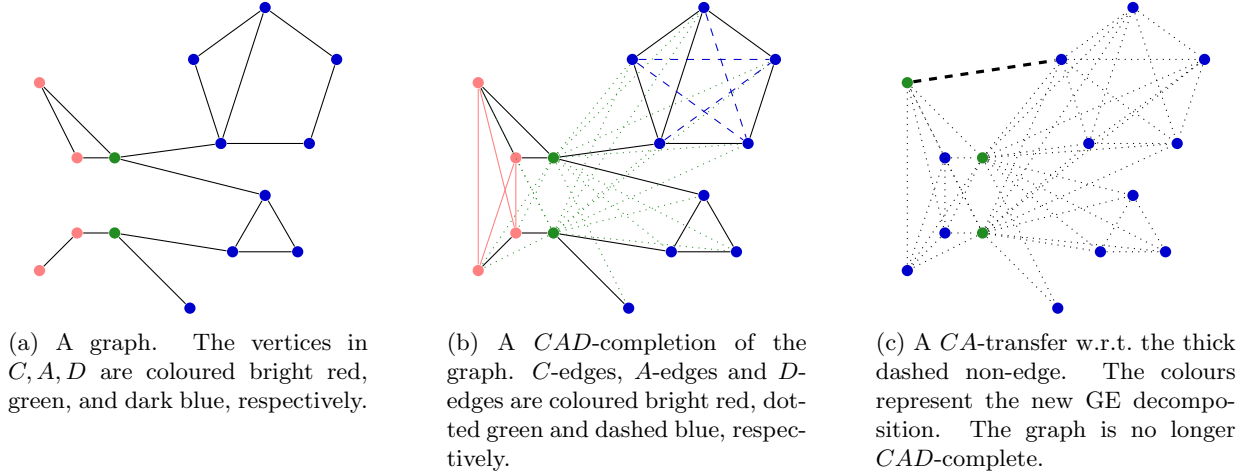
\begin{figure*}[t!]
  \captionsetup{width=0.879\textwidth,font=small}
  \centering
  \begin{subfigure}[t]{0.29\textwidth}
    \centering
    \begin{tikzpicture}[
      every node/.style={fill,color=black,circle,inner sep=1.5pt},
      C/.style={color=red!50!white},
      A/.style={color=ForestGreen},
      D/.style={color=blue!80!black}]
      \node[C] (0) at (1.50000000000000, 1.00000000000000) {};
      \node[C] (1) at (1.50000000000000, 0.000000000000000) {};
      \node[C] (2) at (1.00000000000000, -0.500000000000000) {};
      \node[C] (3) at (1.00000000000000, 2.00000000000000) {};
      \node[A] (4) at (2.00000000000000, 1.00000000000000) {};
      \node[A] (5) at (2.00000000000000, 0.000000000000000) {};
      \node[D] (6) at (3.41221474770753, 1.19098300562505) {};
      \node[D] (7) at (3.04894348370485, 2.30901699437495) {};
      \node[D] (8) at (4.00000000000000, 3.00000000000000) {};
      \node[D] (9) at (4.95105651629515, 2.30901699437495) {};
      \node[D] (10) at (4.58778525229247, 1.19098300562505) {};
      \node[D] (11) at (4.00000000000000, 0.500000000000000) {};
      \node[D] (12) at (3.56698729810778, -0.250000000000000) {};
      \node[D] (13) at (4.43301270189222, -0.250000000000000) {};
      \node[D] (14) at (3.00000000000000, -1.00000000000000) {};
      \draw (0)--(3);
      \draw (0)--(4);
      \draw (1)--(2);
      \draw (1)--(5);
      \draw (3) -- (4);
      \draw (4)--(6);
      \draw (4)--(11);
      \draw (5)--(12);
      \draw (5)--(14);
      \draw (6)--(7);
      \draw (6)--(8);
      \draw (6)--(10);
      \draw (7)--(8);
      \draw (8)--(9);
      \draw (9)--(10);
      \draw (11)--(12);
      \draw (11)--(13);
      \draw (12)--(13);
    \end{tikzpicture}
    \caption{A graph.
    The vertices in $C,A,D$ are coloured bright red, green, and dark blue, respectively.}
    \label{fig:graph}
  \end{subfigure}%
  ~
  \begin{subfigure}[t]{0.04\textwidth}
    ~
  \end{subfigure}
  \begin{subfigure}[t]{0.29\textwidth}
    \centering
    \begin{tikzpicture}[
      every node/.style={fill,color=black,circle,inner sep=1.5pt},
      C/.style={color=red!50!white},
      A/.style={color=ForestGreen},
      D/.style={color=blue!80!black}]
      \node[C] (0) at (1.50000000000000, 1.00000000000000) {};
      \node[C] (1) at (1.50000000000000, 0.000000000000000) {};
      \node[C] (2) at (1.00000000000000, -0.500000000000000) {};
      \node[C] (3) at (1.00000000000000, 2.00000000000000) {};
      \node[A] (4) at (2.00000000000000, 1.00000000000000) {};
      \node[A] (5) at (2.00000000000000, 0.000000000000000) {};
      \node[D] (6) at (3.41221474770753, 1.19098300562505) {};
      \node[D] (7) at (3.04894348370485, 2.30901699437495) {};
      \node[D] (8) at (4.00000000000000, 3.00000000000000) {};
      \node[D] (9) at (4.95105651629515, 2.30901699437495) {};
      \node[D] (10) at (4.58778525229247, 1.19098300562505) {};
      \node[D] (11) at (4.00000000000000, 0.500000000000000) {};
      \node[D] (12) at (3.56698729810778, -0.250000000000000) {};
      \node[D] (13) at (4.43301270189222, -0.250000000000000) {};
      \node[D] (14) at (3.00000000000000, -1.00000000000000) {};
      \draw (0)--(3);
      \draw (0)--(4);
      \draw (1)--(2);
      \draw (1)--(5);
      \draw (3) -- (4);
      \draw (4)--(6);
      \draw (4)--(11);
      \draw (5)--(12);
      \draw (5)--(14);
      \draw (6)--(7);
      \draw (6)--(8);
      \draw (6)--(10);
      \draw (7)--(8);
      \draw (8)--(9);
      \draw (9)--(10);
      \draw (11)--(12);
      \draw (11)--(13);
      \draw (12)--(13);
      % C-edges
      \draw[C] (0)--(1);
      \draw[C] (0)--(2);
      \draw[C] (1)--(3);
      \draw[C] (2)--(3);
      % A-edges
      \draw[A,dotted] (4)--(1);
      \draw[A,dotted] (4)--(2);
      \draw[A,dotted] (4)--(7);
      \draw[A,dotted] (4)--(8);
      \draw[A,dotted] (4)--(9);
      \draw[A,dotted] (4)--(10);
      \draw[A,dotted] (4)--(12);
      \draw[A,dotted] (4)--(13);
      \draw[A,dotted] (4)--(14);
      \draw[A,dotted] (5)--(0);
      \draw[A,dotted] (5)--(2);
      \draw[A,dotted] (5)--(3);
      \draw[A,dotted] (5)--(6);
      \draw[A,dotted] (5)--(7);
      \draw[A,dotted] (5)--(8);
      \draw[A,dotted] (5)--(9);
      \draw[A,dotted] (5)--(10);
      \draw[A,dotted] (5)--(11);
      \draw[A,dotted] (5)--(13);
      % D-edges
      \draw[D,dashed] (6)--(9);
      \draw[D,dashed] (7)--(9);
      \draw[D,dashed] (7)--(10);
      \draw[D,dashed] (8)--(10);
    \end{tikzpicture}
    \caption{A $CAD$-completion of the graph.
    $C$-edges, $A$-edges and $D$-edges
    are coloured bright red, dotted green and dashed blue, respectively.}
    \label{fig:CAD:comp}
  \end{subfigure}
  \begin{subfigure}[t]{0.04\textwidth}
  ~
  \end{subfigure}
  \begin{subfigure}[t]{0.29\textwidth}
    \centering
    \begin{tikzpicture}[
      every node/.style={fill,color=black,circle,inner sep=1.5pt},
      C/.style={color=red!50!white},
      A/.style={color=ForestGreen},
      D/.style={color=blue!80!black}]
      \node[D] (0) at (1.50000000000000, 1.00000000000000) {};
      \node[D] (1) at (1.50000000000000, 0.000000000000000) {};
      \node[D] (2) at (1.00000000000000, -0.500000000000000) {};
      \node[A] (3) at (1.00000000000000, 2.00000000000000) {};
      \node[A] (4) at (2.00000000000000, 1.00000000000000) {};
      \node[A] (5) at (2.00000000000000, 0.000000000000000) {};
      \node[D] (6) at (3.41221474770753, 1.19098300562505) {};
      \node[D] (7) at (3.04894348370485, 2.30901699437495) {};
      \node[D] (8) at (4.00000000000000, 3.00000000000000) {};
      \node[D] (9) at (4.95105651629515, 2.30901699437495) {};
      \node[D] (10) at (4.58778525229247, 1.19098300562505) {};
      \node[D] (11) at (4.00000000000000, 0.500000000000000) {};
      \node[D] (12) at (3.56698729810778, -0.250000000000000) {};
      \node[D] (13) at (4.43301270189222, -0.250000000000000) {};
      \node[D] (14) at (3.00000000000000, -1.00000000000000) {};
      \draw[dotted] (0)--(3);
      \draw[dotted] (0)--(4);
      \draw[dotted] (1)--(2);
      \draw[dotted] (1)--(5);
      \draw[dotted] (3) -- (4);
      \draw[dotted] (4)--(6);
      \draw[dotted] (4)--(11);
      \draw[dotted] (5)--(12);
      \draw[dotted] (5)--(14);
      \draw[dotted] (6)--(7);
      \draw[dotted] (6)--(8);
      \draw[dotted] (6)--(10);
      \draw[dotted] (7)--(8);
      \draw[dotted] (8)--(9);
      \draw[dotted] (9)--(10);
      \draw[dotted] (11)--(12);
      \draw[dotted] (11)--(13);
      \draw[dotted] (12)--(13);
      % C-edges
      \draw[dotted] (0)--(1);
      \draw[dotted] (0)--(2);
      \draw[dotted] (1)--(3);
      \draw[dotted] (2)--(3);
      % A-edges
      \draw[dotted] (4)--(1);
      \draw[dotted] (4)--(2);
      \draw[dotted] (4)--(7);
      \draw[dotted] (4)--(8);
      \draw[dotted] (4)--(9);
      \draw[dotted] (4)--(10);
      \draw[dotted] (4)--(12);
      \draw[dotted] (4)--(13);
      \draw[dotted] (4)--(14);
      \draw[dotted] (5)--(0);
      \draw[dotted] (5)--(2);
      \draw[dotted] (5)--(3);
      \draw[dotted] (5)--(6);
      \draw[dotted] (5)--(7);
      \draw[dotted] (5)--(8);
      \draw[dotted] (5)--(9);
      \draw[dotted] (5)--(10);
      \draw[dotted] (5)--(11);
      \draw[dotted] (5)--(13);
      % D-edges
      \draw[dotted] (6)--(9);
      \draw[dotted] (7)--(9);
      \draw[dotted] (7)--(10);
      \draw[dotted] (8)--(10);
      % CA-transfer edge
      \draw[dashed,very thick] (3)--(7);
    \end{tikzpicture}
    \caption{A $CA$-transfer w.r.t.~the thick dashed non-edge.
    The colours represent the new GE decomposition.
    The graph is no longer $CAD$-complete.}
    \label{fig:CA:trans}
  \end{subfigure}
  \caption{$CAD$-completion and $CA$-transfer.}
  \label{fig:CAD:CA}
\end{figure*}

In this subsection we describe the $AD$-completion algorithm
used in Stage~1 of the main algorithm; see \cref{alg:AD:comp}.
Our first subroutine is $CAD$-completion of $G$ with $\GE(G)=(C,A,D)$, which adds all edges
contained in $C$, incident to $A$, or contained in a connected component of $G[D]$;
see \cref{alg:CAD:comp,fig:CAD:comp}.
\begin{algorithm}
  \caption{$CAD$-completion}\label{alg:CAD:comp}
\begin{algorithmic}
  \Procedure{CAD-complete}{$G=(V,E)$}
    \State $(C,A,D)\gets \GE(G)$
    \State $E \gets E\cup\{\{u,v\}\ :\ u,v\in C$\}
    \State $E \gets E\cup\{\{u,v\}\ :\ u\in A,\ v\in V$\}
    \For{connected component $K$ in $G[D]$}
      \State $E \gets E\cup\{\{u,v\}\ :\ u,v\in K$\}
    \EndFor
    \State \Return $G$
  \EndProcedure
\end{algorithmic}
\end{algorithm}

We will show (see \Cref{cor:CAD:comp}) that CAD-completion 
preserves the GE-decomposition, so 
$G'=\Call{CAD-complete}{G}$ is \defn{$CAD$-complete},
meaning that \Call{CAD-complete}{$G'$}$=G'$. 

\begin{lemma}[Essential vertices]\label{lem:ess}
  Let $G=(V,E)$ be a graph and let $\GE(G)=(C,A,D)$.
  Let $G'$ be obtained from $G$ by adding an edge that is incident to $C\cup A$.
  Then $\nu(G')=\nu(G)$.
\end{lemma}

\begin{proof}
  Let $e=\{u,v\}$ be the added edge, and assume $u\in C\cup A$.
  Write $\nu=\nu(G)$ and $\nu'=\nu(G')$.
  Evidently, $\nu\le\nu'\le\nu+1$.
  Let $M'$ be a maximum matching of $G'$.
  If $\nu'=\nu+1$ then $M'$ must contain $e$.
  But this implies that $M=M'\sm e$ is a maximum matching of $G$
  that does not cover $u$,
  contradicting the assumption that $u\in C\cup A$.
  Thus, $\nu'=\nu$.
\end{proof}

\begin{lemma}[$C/A$-edges]\label{lem:CA}
  Let $G=(V,E)$ be a graph and let $\GE(G)=(C,A,D)$.
  Let $G'$ be obtained from $G$ by adding an edge that is contained in $C$
  or is incident to $A$.
  Then $\GE(G')=\GE(G)$ and $\nu(G')=\nu(G)$.
\end{lemma}

\begin{proof}
  Let $e$ be an edge that is either contained in $C$ or is incident to $A$.
  \cref{lem:ess} implies that $\nu(G') = \nu(G)=:\nu$.
  Thus $D\subseteq D'$, as any maximum matching of $G$ missing $w \in D$
  is also a maximum matching of $G'$ missing $w$.

  Write $\delta=(|D|-k_G(D))/2$.
  By \cref{thm:GE} we know that $2\nu=|A\cup C|+|A|+2\delta$.
  Fix a maximum matching $M$ of $G'$.
  Set
    $a=|M\cap E(A,A\cup C)|$,
    $b=|M\cap E(A,D)|$,
    $c=|M\cap E(C)|$,
    and $d=|M\cap E(D)|$.
  Note that $a+b+c+d=\nu$,
  that $b\le|A|$,
  and that $d\le\delta$.
  Let $x$ denote the number of vertices in $A\cup C$ which are covered by $M$.
  Then $x=2a+b+2c=2\nu-b-2d\ge 2\nu-|A|-2\delta=|A\cup C|$,
  meaning that every vertex of $A\cup C$ is essential in $G'$,
  i.e.~$A\cup C\subseteq A'\cup C'$. Combined with $D\subseteq D'$,
  we deduce $A'\cup C'=A\cup C$ and $D=D'$.
  Since $G\subseteq G'$ and $G'$ has no edges between $C$ and $D=D'$,
  we conclude that $C'=C$ and $A'=A$.
\end{proof}

\begin{lemma}[$D$-edges]\label{lem:D}
  Let $G=(V,E)$ be a graph and let $\GE(G)=(C,A,D)$.
  Let $G'$ be obtained from $G$ by adding an edge 
  contained in a connected component of $G[D]$.
  Then $\GE(G')=\GE(G)$,
  and $G'[D]$ and $G[D]$ have the same component structure.
  In particular, $\nu(G')=\nu(G)$.
\end{lemma}

\begin{proof}
  We first show that $\nu(G')=\nu(G)$.
  Let $M'$ be a maximum matching of $G'$.
  It has at most $|A|$ edges from $A$ to $D$,
  so at least $k_G(D)-|A|$ components of $G[D]$ (and of $G'[D]$)
  are not connected to $A$ by $M'$.
  Thus $M'$ leaves at least $k_G(D)-|A|$ vertices unmatched, so
   $2\nu(G')=|M'|\le |V|-(k_G(D)-|A|)=2\nu(G)$ by  \cref{thm:GE}.
  Clearly, $\nu(G')\ge\nu(G)$, so $\nu(G')=\nu(G)$.
  
  Now let $K$ be a connected component of $G[D]$,
  let $u,v\in K$ and let $G'$ be obtained from $G$ by adding the edge $e=\{u,v\}$.
  Write $\GE(G')=(C',A',D')$.
  As in the proof of \cref{lem:CA}, $\nu(G')=\nu(G)$ implies $D\subseteq D'$
  and it suffices to show $D'\subseteq D$ to deduce $\GE(G')=\GE(G)$.
  
  Let $z\in D'$ and let $M_z$ be a maximum matching of $G'$ missing $z$.
  As $K$ is factor-critical, $K \sm \{z\}$ has a perfect matching, 
  so we may assume that $M_z$ does not contain $e$.
  Then $M_z$ is a maximum matching of $G$ missing $z$, so $D'\subseteq D$.
\end{proof}

As \Call{CAD-complete}{} only adds edges,
we have the following corollary of \cref{lem:CA,lem:D}.
\begin{corollary}[$CAD$-completion]\label{cor:CAD:comp}
  For every graph $G$,
  if $G'=\Call{CAD-complete}{G}$
  then
  $G\subseteq G'$,
  $\GE(G')=\GE(G)$,
  and $\nu(G')=\nu(G)$.
\end{corollary}

Our second subroutine adds a single edge between $C,D$,
assuming they are both non-empty;
see\footnote{%
In the pseudocode, ``{\bf assert} (condition)''
indicates an assumption or precondition that is required to hold at that point;
it is not an operation of the algorithm.}
\cref{alg:CA:trans,fig:CA:trans}.
Note that this operation might output a graph which is not $CAD$-complete.
\begin{algorithm}
  \caption{$CA$-transfer}\label{alg:CA:trans}
\begin{algorithmic}
  \Procedure{CA-transfer}{$G=(V,E)$, $u$, $v$}{}
    \State $(C,A,D)\gets \GE(G)$
    \State {\bf assert} $u\in C$, $v\in D$
    \State $E \gets E\cup\{\{u,v\}\}$
    \State \Return $G$
  \EndProcedure
\end{algorithmic}
\end{algorithm}
We now show that a successful $CA$-transfer empties $C$.
\begin{lemma}[$CA$-transfer]\label{lem:CA:trans}
  Let $G=(V,E)$ be a $CAD$-complete graph and let $\GE(G)=(C,A,D)$.
  Assume $C,D\ne\es$ and let $u\in C$ and $v\in D$.
  Let $G'=\Call{CA-transfer}{G,u,v}$ and write $\GE(G')=(C',A',D')$.
  Then $G\subseteq G'$,
  $C'=\es$, $A'=A\cup\{u\}$, and $\nu(G')=\nu(G)$.
\end{lemma}

\begin{proof}
  As in \cref{lem:CA}, 
  we have $\nu(G')=\nu(G)$ and $D\subseteq D'$.
  Let $z\in C\sm\{u\}$ and  $M$ be a maximum matching of $G$ 
  missing $v$. As $G$ is $CAD$-complete, we can assume $\{u,z\} \subseteq M$.
  Let $M'=M\sm\{\{z,u\}\}\cup\{\{u,v\}\}$.
  Then $M'$ is a maximum matching of $G'$ missing $z$,
  so $C\sm\{u\}\subseteq D'$. 
  
  It remains to show  $A \cup \{u\} \subseteq A'$.
  Any $x \in A$ has a  neighbour (in $G$) 
  in $D\subseteq D'$, so $A \subseteq A'$.
  Also,  $u\notin D'$,
  as every maximum matching of $G$ contains $u$,
  and every maximum matching of $G'$ is either a maximum matching of $G$ 
  or uses $\{u,v\}$, so contains $u$ either way. 
  Furthermore, $u \notin C'$ as $u$ is adjacent in $G'$ to $v \in D \subseteq D'$.
  Thus $u \in A'$.
 \end{proof}

We conclude this subsection with AD-completion,
which combines the previous two subroutines
to return a $CAD$-complete graph with empty $C$ (unless $C=V$);
see \cref{alg:AD:comp}.
(An additional $CAD$-completion may be 
necessary after a CA-transfer, see \cref{fig:CA:trans}.)
\begin{algorithm}
  \caption{$AD$-completion}\label{alg:AD:comp}
\begin{algorithmic}
  \Procedure{AD-complete}{$G=(V,E)$}
    \State $G\gets$ \Call{CAD-complete}{$G$}
    \Comment \cref{alg:CAD:comp}
    \State $(C,A,D)\gets \GE(G)$
    \If{$C=\es$ or $D=\es$}
      \State \Return $G$
    \EndIf
    \State Let $u\in C$, $v\in D$
    \State $G\gets$ \Call{CA-transfer}{$G,u,v$}
    \Comment \cref{alg:CA:trans}
    \State $G\gets$ \Call{CAD-complete}{$G$}
    \Comment \cref{alg:CAD:comp}
    \State \Return $G$
  \EndProcedure
\end{algorithmic}
\end{algorithm}

Say that $G$ with $\GE(G)=(C,A,D)$ is \defn{$AD$-complete} if it is $CAD$-complete and $C=\es$.
For future reference, we observe that a graph $G$ with vertex partition $(A,D)$
is $AD$-complete with $\GE(G)=(\es,A,D)$ if and only if every vertex in $A$
is adjacent to all other vertices and $G[D]$ is the disjoint union
of more than $|A|$ cliques of odd size. This implies the following lemma.

\begin{lemma}[$D$-flation]\label{lem:D:flation}
  Let $G=(V,E)$ be an $AD$-complete graph with $\GE(G)=(\es,A,D)$.
  Let $K$ be a connected component of $G[D]$ of (odd) size $\gamma$.
  Obtain $G'$ from $G$ by replacing $K$ with a clique $K'$ 
  of some odd size $\gamma'\ge 1$
  (adding all edges between $A$ and $K'$).
  Then $\GE(G')=(\es,A,D')$
  for $D'=(D\sm K)\cup K'$,
  $k_{G'}(D')=k_G(D)$,
  $\nu(G')=\nu(G)+(\gamma'-\gamma)/2$,
  and $G'$ is $AD$-complete.
\end{lemma}

We conclude this subsection with
the following consequence of \cref{cor:CAD:comp,lem:CA:trans},
noting that $D=\es$ implies $C=V$ and hence $\nu(G)=|V|/2$.
\begin{corollary}[$AD$-completion]\label{cor:AD:comp}
  Let $G=(V,E)$ be a graph with $\nu(G)<|V|/2$,
  and let $G'=\Call{AD-complete}{G}$.
  Then
  $G\subseteq G'$,
  $\nu(G')=\nu(G)$ and $G'$ is $AD$-complete.
\end{corollary}

\subsection{Isolation}\label{sec:isolation}

Here we present the clique isolation algorithm 
(see \cref{alg:D:isol,fig:D:isol}), which is used to peel a leaf
by the dissolving method in Stage~3 of the main algorithm.
Its analysis will require
$G$ to be \defn{$D$-complete} (see \cref{fig:graph:D:comp}),
i.e.~a disjoint union of odd cliques, so $\GE(G)=(\es,\es,D)$.
We also require $K$ to be \defn{scattered}, 
meaning that its vertices all belong to distinct cliques.

\begin{figure*}[t!]
  \captionsetup{width=0.879\textwidth,font=small}
  \centering
  \begin{subfigure}[t]{0.29\textwidth}
    \centering
    \begin{tikzpicture}[
      every node/.style={fill,color=black,circle,inner sep=1.5pt},
      C/.style={color=red!50!white},
      A/.style={color=ForestGreen},
      D/.style={color=blue!80!black}]

      \clip (1,-0.45) rectangle (5.5,4);

      \node[D] (6) at (3.41221474770753, 1.19098300562505) {};
      \node[D] (7) at (3.04894348370485, 2.30901699437495) {};
      \node[D] (8) at (4.00000000000000, 3.00000000000000) {};
      \node[D] (9) at (4.95105651629515, 2.30901699437495) {};
      \node[D] (10) at (4.58778525229247, 1.19098300562505) {};
      \node[D] (11) at (3.00000000000000, 0.500000000000000) {};
      \node[D] (12) at (2.56698729810778, -0.250000000000000) {};
      \node[D] (13) at (3.43301270189222, -0.250000000000000) {};
      \node[D] (14) at (1.84683045111455, 2.60901699437495) {};
      \node[D] (15) at (1.84683045111455, 1.49098300562505) {};
      \node[D] (16) at (1.84683045111455, 0.372949016875150) {};
      \draw (6)--(7);
      \draw (6)--(8);
      \draw (6)--(10);
      \draw (7)--(8);
      \draw (8)--(9);
      \draw (9)--(10);
      \draw (11)--(12);
      \draw (11)--(13);
      \draw (12)--(13);
      \draw (6)--(9);
      \draw (7)--(9);
      \draw (7)--(10);
      \draw (8)--(10);
    \end{tikzpicture}
    \caption{A $D$-complete graph.}
    \label{fig:graph:D:comp}
  \end{subfigure}%
  ~
  \begin{subfigure}[t]{0.04\textwidth}
    ~
  \end{subfigure}
  \begin{subfigure}[t]{0.29\textwidth}
    \centering
    \begin{tikzpicture}[
      vx/.style={fill,color=black,circle,inner sep=1.5pt},
      C/.style={color=red!50!white},
      A/.style={color=ForestGreen},
      D/.style={color=blue!80!black}]

      \clip (1,-0.45) rectangle (5.5,4);

      \node[vx] (6) at (3.41221474770753, 1.19098300562505) {};
      \node[vx] (7) at (3.04894348370485, 2.30901699437495) {};
      \node[vx] (8) at (4.00000000000000, 3.00000000000000) {};
      \node[vx] (9) at (4.95105651629515, 2.30901699437495) {};
      \node[vx] (10) at (4.58778525229247, 1.19098300562505) {};
      \node[vx] (11) at (3.00000000000000, 0.500000000000000) {};
      \node[vx] (12) at (2.56698729810778, -0.250000000000000) {};
      \node[vx] (13) at (3.43301270189222, -0.250000000000000) {};
      \node[vx] (14) at (1.84683045111455, 2.60901699437495) {};
      \node[vx] (15) at (1.84683045111455, 1.49098300562505) {};
      \node[vx] (16) at (1.84683045111455, 0.372949016875150) {};
      \draw (11)--(12);
      \draw (11)--(13);
      \draw (12)--(13);

    \node[fit={($(6)-(0,0.2)$)($(7)-(0.2,0)$)($(8)+(0,0.2)$)($(9)+(0.2,0)$)(10)},
          rectangle,rounded corners,fill,black,opacity=0.085] (L) {};
    \node[black,above] at (L.north) {$L$};
    \node[fit={($(9)+(0.1,0.1)$)($(10)-(0.1,0.1)$)},
          rectangle,rounded corners,fill,ForestGreen,opacity=0.17] (S) {};
    \node[ForestGreen,above] at (S.north) {$S$};
    \node[fit={($(14)+(0.2,0.2)$)($(16)-(0.2,0.2)$)(7)},
          rectangle,rounded corners,fill,purple,opacity=0.17] (K) {};
    \node[purple,above] at (K.north west) {$K$};

    \draw[D] (9) -- (10);
    \draw[D] (9) -- (7) -- (10);
    \draw[D] (9) -- (11) -- (10);
    \draw[D] (9) -- (14) -- (10);
    \draw[D] (9) -- (15) -- (10);
    \draw[D] (9) -- (16) -- (10);

    \end{tikzpicture}
    \caption{
    Isolation of a clique, $L$,
    with respect to its subset $S$
    and a disjoint, scattered set $K$.
    }
    \label{fig:D:isol:step}
  \end{subfigure}
  \begin{subfigure}[t]{0.04\textwidth}
  ~
  \end{subfigure}
  \begin{subfigure}[t]{0.29\textwidth}
    \centering
    \begin{tikzpicture}[
      vx/.style={fill,color=black,circle,inner sep=1.5pt},
      C/.style={color=red!50!white},
      A/.style={color=ForestGreen},
      D/.style={color=blue!80!black}]

      \clip (1,-0.45) rectangle (5.5,4);

      \node[vx,D] (6) at (3.41221474770753, 1.19098300562505) {};
      \node[vx,D] (7) at (3.04894348370485, 2.30901699437495) {};
      \node[vx,D] (8) at (4.00000000000000, 3.00000000000000) {};
      \node[vx,A] (9) at (4.95105651629515, 2.30901699437495) {};
      \node[vx,A] (10) at (4.58778525229247, 1.19098300562505) {};
      \node[vx,D] (11) at (3.00000000000000, 0.500000000000000) {};
      \node[vx,D] (12) at (2.56698729810778, -0.250000000000000) {};
      \node[vx,D] (13) at (3.43301270189222, -0.250000000000000) {};
      \node[vx,D] (14) at (1.84683045111455, 2.60901699437495) {};
      \node[vx,D] (15) at (1.84683045111455, 1.49098300562505) {};
      \node[vx,D] (16) at (1.84683045111455, 0.372949016875150) {};
      \draw (11)--(12);
      \draw (11)--(13);
      \draw (12)--(13);

    \node[fit={($(12)-(0,0.1)$)($(14)-(0.1,0)$)($(8)+(0.1,0.1)$)},
          rectangle,rounded corners,fill,blue,opacity=0.17] (D) {};
    \node[blue,above] at (D.north) {$D$};
    \node[fit={($(9)+(0.1,0.1)$)($(10)-(0.1,0.1)$)},
          rectangle,rounded corners,fill,ForestGreen,opacity=0.17] (A) {};
    \node[ForestGreen,above] at (S.north) {$A$};

    \draw (9) -- (10);
    \draw (9) -- (7) -- (10);
    \draw (9) -- (11) -- (10);
    \draw (9) -- (14) -- (10);
    \draw (9) -- (15) -- (10);
    \draw (9) -- (16) -- (10);

    \end{tikzpicture}
    \caption{The GE-decomposition of the resulting graph.
      The resulting graph need not be $AD$-complete.}
    \label{fig:D:isol:after}
  \end{subfigure}
  \caption{$D$-isolation (\cref{alg:D:isol}).}
  \label{fig:D:isol}
\end{figure*}

\begin{algorithm}
  \caption{$D$-isolation}\label{alg:D:isol}
\begin{algorithmic}
  \Procedure{D-isolate}{$G=(V,E)$, $L$, $S$, $K$}
    \State $(C,A,D)\gets \GE(G)$
    \State {\bf assert} $C=A=\es$
    \State {\bf assert} $S\subseteq L$ and $K\cap S=\es$
    \State $\kappa\gets(|L|-1)/2$
    \State {\bf assert} $|K|>|S|=\kappa$
    \State $E\gets E\sm E(L)$
    \State $E\gets E\cup E(S) \cup E(S,K)$
    \State \Return $G$
  \EndProcedure
\end{algorithmic}
\end{algorithm}

\begin{lemma}[$D$-isolation]\label{lem:D:isol}
  Let $G=(V,E)$ be a $D$-complete graph with $\GE(G)=(\es,\es,D)$.
  
  Let $L$ be a maximal clique in $G$ of size $2\kappa+1$ with $\kappa \in \NN_+$,
  let $S\subseteq L$ with $|S|=\kappa$
  and let $K\subseteq V\sm S$ with $|K|>\kappa$ be scattered.
  Write $G'=\Call{D-isolate}{G,L,S,K}$
  and $\GE(G')=(C',A',D')$.
  
  Then
  $C'=\es$,
  $A'=S$,
  $D'=D\sm S$,
  and $\nu(G')=\nu(G)$.
  Moreover, the connected components of $G'[D']$
  are the connected components of $G$
  with $L$ replaced with $\kappa+1$ isolated vertices.
\end{lemma}

\begin{proof}
  We start by showing $\nu(G')=\nu(G)$.
  Form $G_1$ from $G$ by replacing the component $L$ with a single vertex $x$.
  By \cref{lem:D:flation}, we have $\nu(G_1)=\nu(G)-\kappa$
  and $k_{G_1}(D)=k_G(D)$. 
  Now form $G_0$ from $G_1$ by replacing $x$ with the vertices of $L$
  (without adding any edges).
  Evidently, $\nu(G_0)=\nu(G_1)=\nu(G)-\kappa$,
  $\GE(G_0)=\GE(G)=(\es,\es,D)$,
  and $k_{G_0}(D)=k_{G_1}(D)+2\kappa=k_G(D)+2\kappa$.
  Write $G' = G_0\cup F$, where
  $F$ consists of all  possible edges within $S$ or between $S$ and $K$.
  As $|K|>|S|$, we have $\tau(F)=|S|=\kappa$.
  Also, all pairs in $F$ are non-edges of $G_0$,
  as $G$ is a disjoint union of cliques, one of which is $L$,
  so $E(G) \cap F \subseteq E(L)$.
  We deduce $\nu(G')\le\nu(G_0)+\kappa\le\nu(G)$ by \cref{lem:cover}.
  For the other direction, we construct a matching of size $\nu(G)$ in $G'$.
  As $G$ is a disjoint union of odd cliques and $K$ is scattered
  we can choose a maximum matching $M$ of $G$ that misses $K$.
  Then $M_0 := M \sm L$ is a maximum matching of $G_0$
  of size $\nu(G)-\kappa$ that misses $K\cup S$.
  Now the required matching in $G'$ of size $\nu(G)$
  is $M' := M_0 \cup M_S$ where $M_S$
  is a matching between $S$ and $K$ that saturates $S$.
  Thus $\nu(G')=\nu(G)$, as claimed.
  Furthermore, by \cref{lem:cover} again, 
  every vertex of $S$ is essential in $G'$, so $S \cap D'=\es$.
  
  To complete the proof, it suffices to show $D \sm S \subseteq D'$.
  Indeed, then every vertex in $S$ has a neighbour in $K \subseteq D'$,
  so $S \subseteq A'$, giving $D'=D\sm S$, $A'=S$ and $C'=\es$.
  We consider any $v \in D \sm S$ and show $v \in D'$.
  If $v\in K$, as $|K|>|S|$ we can construct $M'$ above to miss $v$, so $v \in D'$.
  It remains to consider $v\in D\sm(S\cup K)$.
  We are done if $M'$ misses $v$, so suppose $M'$ must cover $v$,
  meaning that $v$ belongs to some clique $C$ 
  which also contains some (unique) $u \in K$.
  We modify $M'$ by replacing the near-perfect matching of $C$
  by one that misses $v$ (so covers $u$ instead),
  obtaining a maximum matching of $G'$ missing $v$, as required.
\end{proof} 

For future reference we also record the following
obvious property of $D$-complete graphs.
\begin{lemma}[Maximal cliques]\label{lem:K}
  Let $G=(V,E)$ be a $D$-complete graph and let $\GE(G)=(\es,\es,D)$.
  Let $G'=G[V \sm K]$ for some maximal clique $K$.
  Then $\GE(G')=(\es,\es,D\sm K)$
  and $\nu(G')=\nu(G)-\lfloor{|K|/2}\rfloor$.
  In particular, if $|K|=1$ then $\nu(G')=\nu(G)$.
\end{lemma}

\subsection{Merging}\label{sec:merging}

We conclude this section with the clique merging algorithm 
(see \cref{alg:D:merge}), which is used to peel a leaf
by the merging method in Stage~3 of the main algorithm.

\begin{algorithm}
  \caption{$D$-merging}\label{alg:D:merge}
\begin{algorithmic}
  \Procedure{D-merge}{$G=(V,E)$, $L$, $w$, $K$}
    \State {\bf assert} $L$ is a maximal clique in $G$
    \State {\bf assert} $K$ is a maximal clique in $G$
    \State $L'\gets L\sm\{w\}$
    \State {\bf assert} $L'\cap K=\es$
    \State $E\gets E\sm \{\{w,v\}\ :\ v\in L'\}$
    \State $E\gets E\cup\{\{u,v\}\ :\ u\in L',\ v\in K\}$
    \State \Return $G$
  \EndProcedure
\end{algorithmic}
\end{algorithm}
\begin{lemma}[$D$-merging]\label{lem:D:merge}
  Let $G=(V,E)$ be a $D$-complete graph with $\GE(G)=(\es,\es,D)$.
  Let $L,K$ be two distinct nontrivial maximal cliques in $G$.
  Let $w\in L$,
  and set $G'=\Call{D-merge}{G,L,w,K}$.
  Then $\GE(G')=\GE(G)$, $k_{G'}(D)=k_G(D)$, and $G'$ is $D$-complete.
  In particular, $\nu(G')=\nu(G)$.
\end{lemma}

\begin{proof}
  We apply \cref{lem:D:flation} twice to deflate $L,K$
  into {\em distinct} isolated vertices $w\in L$ and $u\in K$.
  Denote the resulting graph by $G_\bullet$.
  Observe that $\nu(G_\bullet)=\nu(G)+(2-|L|-|K|)/2$.
  Now, inflate $G_\bullet$ again by replacing $u$ with a clique on $L'\cup K$,
  where $L'=L\sm\{w\}$.
  Then the resulting graph is $G'$,
  and $\nu(G')=\nu(G_\bullet)+(|L|+|K|-1-1)/2=\nu(G)$.
  We also deduce that $\GE(G')=\GE(G)$, $k_{G'}(D)=k_G(D)$, and $G'$ is $D$-complete.
\end{proof}

\section{The main algorithm} \label{sec:alg}

Besides the single-colour compressions described in the previous section,
our main algorithm also requires more intricate multicolour compressions.
In \Cref{sec:decycle} we present the optimisation procedure 
used for decycling in Stage~2 of the main algorithm. The remaining coloured
graph has a forest-like structure of cliques, which is exploited in \Cref{sec:peel} 
for the peeling procedure for removing leaves. In \Cref{sec:distil} we present
the main algorithm {\em (distilling)} and deduce our structural result \Cref{thm:max},
deferring \Cref{thm:grt} to the next section.

\begin{figure}[t!]
\captionsetup{width=0.879\textwidth,font=small}
  \centering
  % Hypergraph
  \begin{subfigure}[t]{0.48\textwidth}
    \centering
      \includegraphics[width=\textwidth]{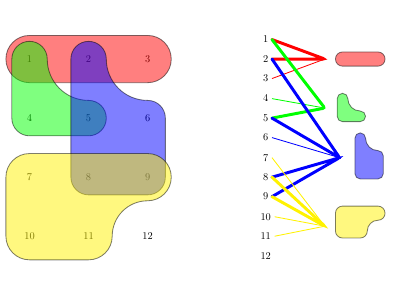}
      \subcaption{A hypergraph and its cyclic incidence graph.
      The cycles are emphasised with thicker lines.}
      \label{fig:hypergraph}
  \end{subfigure}
  \hfill
  % Incidence graph
  \begin{subfigure}[t]{0.48\textwidth}
    \centering
      \includegraphics[width=\textwidth]{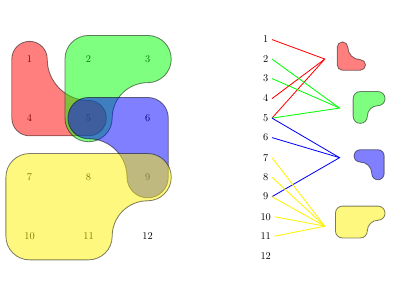}
    \caption{A hyperforest and its incidence graph.
    The red, green and yellow edges are leaves,
    with links $5$, $5$ and $9$.}
      \label{fig:hyperforest}
  \end{subfigure}
  
  \caption{Incidence graphs.}
  \label{fig:incidence:graph}
\end{figure}

\subsection{Decycling} \label{sec:decycle}

In this subsection we implement Stage~2 of the main algorithm,
the {\em decycling} procedure; see \cref{alg:decycle,fig:decycle}.
As discussed in the proof outline, for some $T \subseteq V$
we will delete all coloured edges incident to $T$
and add all possible uncoloured edges incident to $T$,
so that (a) the remaining coloured cliques form a {\em hyperforest},
and (b) the sum of matching numbers over all colours decreases by at least $|T|$.
To prepare for the choice of $T$, we first need to define hyperforests
and formulate an appropriate optimisation problem
that guarantees the required properties.

For a hypergraph $\cH=(U,\vect{H})$,
we define the \defn{incidence graph} of $\cH$,
denoted $I_\cH$, as follows.
The vertex set of $I_\cH$ is $U\cup\vect{H}$,
partitioned into two parts $U$ and $\vect{H}$.
A vertex pair $\{u,S\}$
with $u\in U$ and $S\in\vect{H}$
is connected by an edge if and only if $u\in S$.
We say that $\cH$ is a (loose) \defn{hyperforest}
if $I_\cH$ is a forest (i.e.,~has no cycles).
A \defn{leaf-edge} in a hyperforest $\cH=(U,\vect{H})$
is an edge $L\in\vect{H}$
for which all but at most one neighbour of $L$ in $I_\cH$ are leaves.
If such a non-leaf neighbour $w$ exists, we call it the \defn{link} of that leaf-edge.
See \cref{fig:incidence:graph} for a couple of examples.
We observe that if $|\vect{H}|\ge 2$ then $\cH$ contains at least two leaf-edges,
and that by removing a leaf-edge, the property of being a hyperforest is preserved.
For simpler terminology, we also refer to a leaf-edge as a \defn{leaf}.
For a leaf $L\in\cH$,
let $\link(L)$ be the link of $L$ in $\cH$, if such exists,
or an arbitrary vertex of $L$ otherwise.

Now we formulate the appropriate optimisation problem.
For two hypergraphs $\cX=(U,\vect{X})$ and $\cY=(U,\vect{Y})$ on the same vertex set,
define $\sigma=\sigma_{\cX,\cY}:\cP(U)\to\ZZ$ as follows:
\[
  \sigma(T)=r(T)-|T|, \quad \text{where} \quad
    r(T) = \sum_{X\in\vect{X}}\floor{|T\cap X|/2}
    +\sum_{Y\in\vect{Y}}|T\cap Y|.
\]
We say that $T\subseteq U$ is \defn{$\sigma$-maximal}
if for every $S\subseteq U$ we have $\sigma(S)\le\sigma(T)$,
and for every $T'\supsetneq T$ we have $\sigma(T')<\sigma(T)$.
Note that since $\sigma(\es)=0$,
a $\sigma$-maximal set $T$ satisfies $\sigma(T)\ge 0$.

\begin{lemma}[$\sigma$-maximal sets]\label{lem:sigma}
  Let $\cX=(U,\vect{X})$ and $\cY=(U,\vect{Y})$ be hypergraphs,
  and let $\sigma=\sigma_{\cX,\cY}$.
  Suppose $T\subseteq U$ is $\sigma$-maximal.
  Then
  \begin{enumerate}
    \item\label{it:sig:Y} For every $Y\in\vect{Y}$, $T\supseteq Y$;
    \item\label{it:sig:X} For every $X\in\vect{X}$, either $T\supseteq X$ or $|T\cap X|$ is even;
    \item\label{it:sig:hf} $\cX[U\sm T]$ is a hyperforest.
  \end{enumerate}
\end{lemma}

\begin{proof}
  Suppose first that $Y\sm T\ne\es$ for some $Y\in\vect{Y}$,
  and let $y\in Y\sm T$.
  Set $T_y=T\cup\{y\}$,
  and note that
  $\sigma(T_y)-\sigma(T)\ge -|T_y|+|T|+|T_y\cap Y|-|T\cap Y|\ge 0$,
  a contradiction.
  Similarly, suppose that $X\sm T\ne\es$ and $|T\cap X|$ is odd for some $X\in\vect{X}$,
  and let $x\in X\sm T$.
  Set $T_x=T\cup\{x\}$,
  and note that
  $\sigma(T_x)-\sigma(T)\ge -|T_x|+|T| + \floor{|T_x\cap X|/2}-\floor{|T\cap X|/2}\ge 0$,
  a contradiction.
  Finally, suppose that $I:=I_{\cX[U\sm T]}$,
  contains a cycle $\{u_1,X_1,\dots,u_j,X_j,u_1\}$
  of length $2j$ for some $j\ge 2$,
  with $u_i\in U\sm T$
  and $X_i$ an edge of $\cX[U\sm T]$ for all $i\in[j]$.
  Set $L=\{u_1,\dots,u_j\}$
  and $T^\circ=T\cup L$.
  As $|X_i\cap L|\ge 2$ for all $i\in[j]$,
  we obtain the contradiction
  \[
    \sigma(T^\circ)-\sigma(T)
      \ge -|T^\circ|+|T|
      + \sum_{i=1}^j \left(\floor{|T^\circ\cap X_i|/2}-\floor{|T\cap X_i|/2}\right)
      \ge 0. \qedhere
  \]
\end{proof}

To implement the uncolouring part of the proof strategy,
we extend the notion of coloured graphs to 
allow for a set of uncoloured edges,
which we will denote by $G_0$.
In particular, if $G_0=(V,\es)$ is empty,
we identify $(G_1,\dots,G_q)$ with $(G_0,\dots,G_q)$.
We keep the notation $\nu(\cG) = (\nu(G_1), \dots, \nu(G_q))$,
and let $\mns(\cG)=\|\nu(\cG)\|_1=\sum_{j=1}^q\nu(G_j)$,
so $\mns((G_0,\dots,G_q)) = \mns((G_1,\dots,G_q))$,
although $E =   \bigcup_{j=0}^q E_j$ may differ from
 $\bigcup_{j=1}^q E_j$ due to uncoloured edges.
 
 \begin{definition} (Uncolouring)
For $T\subseteq V$,
the \defn{star neighbourhood} $\nabla_T$ of $T$
is the set of all pairs of vertices that contain at least one vertex of $T$.
For a $q$-colouring $\cG=(G_0,\dots,G_q)$ of $G$,
write $\unclr(\cG;T)=(G_0',\dots,G_q')$,
where $G_j'=(V,E_j')$,
$E_0'=E_0\cup \nabla_T$,
and $E_j'=E_j\sm \nabla_T$ for $j=1,\dots,q$.
(In words, we remove any edge incident to $T$
and connect every vertex of $T$ to all other vertices by uncoloured edges;
see \cref{alg:decycle,fig:decycle}.)
\end{definition}

 The hypergraphs $\cX$ and $\cY$ considered above
 will be obtained from the GE-decompositions
 of the coloured graphs, as in the following definition.
 
\begin{definition} (GE-surplus)
Let $\cG=(G_0,G_1,\dots,G_q)$ be a $q$-colouring of $G=(V,E)$.
For $j\in[q]$,
write $\GE(G_j)=(C_j,A_j,D_j)$,
and
let $K_j^1,\dots,K_j^{\iota_j}$ be the sets of vertices of the nontrivial cliques of $G_j[D_j]$.
Set $\cK=\cK(\cG)=(V,\{K_j^i:j\in[q],\ i\in[\iota_j]\})$;
if this family is empty, choose any $v\in V$
and set instead $\cK=(V,\{\{v\}\})$.
Set further $\cA=\cA(\cG)=(V,\{A_1,\dots,A_q\})$ (we allow $\cA$ to be empty),
and define the \defn{GE-surplus} $\sigma_\cG$ of $\cG$ to be $\sigma_{\cK,\cA}$.
\end{definition}

We require some further definitions before stating the decycling algorithm.

\begin{itemize}[nosep]
\item Say that $\cG$ is \defn{$AD$-complete} if $G_j$ is $AD$-complete for all $j\in[q]$.
\item Say that $\cG$ is \defn{$D$-complete} if $G_j$ is $D$-complete for all $j\in[q]$.
\item Say that $\cG$ is \defn{$D$-acyclic} if it is $D$-complete and $\cK$ is a hyperforest.
\item Let $\Theta(\cG)$ denote the set of vertices of degree $|V|-1$ in $G_0$,
  unless $G_0$ is complete,
  in which case we fix $\Theta(\cG)$ to be a designated subset of $V$ of size at least $|V|-1$.\footnote{When $G_0$ is complete, $\Theta(\cG)$ is not uniquely determined by $\cG$. Whenever this case occurs, we fix a particular choice of $\Theta(\cG)$ (as specified at that point in the proof) and keep that designated choice thereafter.}
\item We say that $\cG$ is \defn{proper} if $E_0$ is disjoint from $E_1,\dots,E_q$
and $\Theta(\cG)$ is a cover for $E_0$.
\item We say that $\cG$ is \defn{$\Theta$-complete} if $\cG[V\sm\Theta(\cG)]$ is $AD$-complete.
\end{itemize}

Note that if $\cG$ is proper and $D$-acyclic then it is also $\Theta$-complete.

\begin{algorithm}
  \caption{Decycling}\label{alg:decycle}
\begin{algorithmic}
  \Procedure{Decycle}{$\cG=(G_0,G_1,\dots,G_q)$}
    \State {\bf assert} $\cG$ is $\Theta$-complete
    \State Let $T\subseteq V$ be $\sigma_\cG$-maximal
    \State \Return $\unclr(\cG;T),T$
  \EndProcedure
\end{algorithmic}
\end{algorithm}

\begin{figure*}[t!]
  \captionsetup{width=0.879\textwidth,font=small}
  \centering
  \begin{subfigure}[t]{0.29\textwidth}
    \centering
    \includegraphics[width=\textwidth]{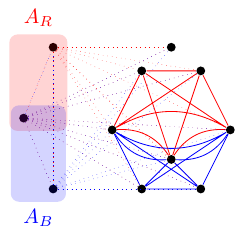}
    \caption{An $AD$-complete $2$-coloured graph.}
    \label{fig:cgraph:AD:comp}
  \end{subfigure}%
  ~
  \begin{subfigure}[t]{0.04\textwidth}
    ~
  \end{subfigure}
  \begin{subfigure}[t]{0.29\textwidth}
    \centering
    \includegraphics[width=\textwidth]{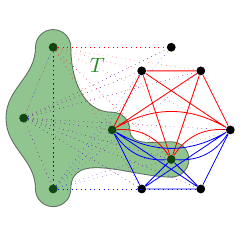}
    \caption{
      A $\sigma$-maximal $T$.
      Here, ${\sigma(T)=1}$.
    }
    \label{fig:sigma:max}
  \end{subfigure}
  \begin{subfigure}[t]{0.04\textwidth}
  ~
  \end{subfigure}
  \begin{subfigure}[t]{0.29\textwidth}
    \centering
    \includegraphics[width=\textwidth]{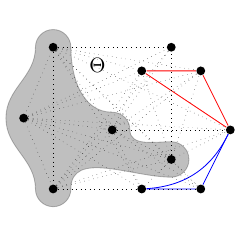}
    \caption{
      The decycled coloured graph.
    }
    \label{fig:decycle:step}
  \end{subfigure}
  \caption{decycling (\cref{alg:decycle}).}
  \label{fig:decycle}
\end{figure*}

Before analysing the decycling procedure, we record a lemma on acyclicity
that will be needed for the analysis of dissolution in the next subsection.

\begin{lemma}[Stability of acyclicity]\label{lem:acyclic:stable}
  Let $\cG$ be a $D$-acyclic $q$-colouring of $G=(V,E)$,
  and let $\sigma=\sigma_{\cG}$.
  Then $\es$ is $\sigma$-maximal (so it is the only $\sigma$-maximal set).
\end{lemma}

\begin{proof}
  We will show that $\sigma(S)<0$ for every non-empty $S\subseteq V$.  
  For $\cH\subseteq\cK$ and $U\subseteq V$ write
  $s_\cH(U)=\sum_{K\in\cH}\floor{|U\cap K|/2}$.
  As $\cG$ is $D$-acyclic, $\cK$ is a hyperforest and $\sigma(S) = -|S| + s_\cK(S)$.
  We will show $s_\cH(S)<|S|$ for any $\cH\subseteq\cK$ by induction on $|\cH|$.
  
  If $|\cH|=1$ then
  $s_\cH(S)\le \floor{|S|/2}<|S|$.
  Assume then that $|\cH|\ge 2$, and let $L$ be a leaf of $\cH$.
  If $L$ intersects any other edge of $\cH$,
  let $w$ be the unique point of intersection, and set $L'=L\sm\{w\}$;
  otherwise, set $L'=L$.
  Write $\cH'=\cH\sm\{L\}$ and
  \[
    -|S| + s_{\cH}(S)
    = (-|S\cap L'| +\floor{|S\cap L|/2})
    + (-|S\sm L'|  +s_{\cH'}(S\sm L')).
  \]
 We can assume $S\sm L'\ne\es$, as otherwise
 $-|S| + s_{\cH}(S) = -|S|+\floor{|S|/2}<0$,
 so $s_{\cH'}(S\sm L')<|S\sm L'|$ by the induction hypothesis.
 It remains to show $-|S\cap L'| +\floor{|S\cap L|/2} \le 0$.
 
 This is immediate if $S\cap L=\es$, 
 or otherwise it holds as
      \[
        -|S\cap L'|+\floor{|S\cap L|/2}
        \le -|S\cap L|+1+\floor{|S\cap L|/2}
        \le 0. \qedhere
      \]  % \end{itemize}
\end{proof}

We conclude this subsection by analysing the decycling procedure.
Here, and in subsequent compressions, we do not increase the matching number 
of any colour or decrease the number of copies of $K_\ell$ in the underlying graph;
if $\cG=(G_0,\dots,G_q)$ is a $q$-colouring of $G$
we let $m_\ell(\cG)$ denote the number of copies of $K_\ell$ in $G$.

Also, as previously discussed, when we add some set $T$ to the uncoloured set,
the sum of all matching numbers should decrease by at least $|T|$.
The remaining conclusions in the lemma keep track of the set $G_0$ of uncoloured edges
so that we maintain a proper $D$-acyclic colouring.

\begin{lemma}[Decycling]\label{lem:decycle}
  Let $\cG$ be a proper $\Theta$-complete $q$-colouring of $G=(V,E)$,
  and let $\cG',T=\Call{Decycle}{\cG}$.
  Then
  \begin{enumerate}
    \item\label{it:dec:ml} $m_\ell(\cG')\ge m_\ell(\cG)$;
    \item\label{it:dec:nu} $\nu(\cG')\le\nu(\cG)$;
    \item\label{it:dec:mns} $\mns(\cG') \le \mns(\cG)-|T|$;
    \item\label{it:dec:disjoint} $\Theta(\cG)\cap T=\es$;
    \item\label{it:dec:theta} $\Theta(\cG')=\Theta(\cG)\cup T$;
    \item\label{it:dec:prop} $\cG'$ is a proper $D$-acyclic $q$-colouring of $G$.
  \end{enumerate}
\end{lemma}

\begin{proof}
  First, since $\unclr(\cG;T)$ is obtained from $\cG$ by adding (and uncolouring) edges,
  we trivially have $m_\ell(\cG')\ge m_\ell(\cG)$;
  this settles (\ref{it:dec:ml}).
  Since every added edge is uncoloured, we also have $\nu(\cG')\le \nu(\cG)$,
  which settles (\ref{it:dec:nu}).

  Let $\sigma=\sigma_\cG$, and recall that $T$ is $\sigma$-maximal.
  Write $\cG'=(G_0',\dots,G_q')$
  and note that $\nu(G_j')=\nu(G_j[V\sm T])$ for $j\in[q]$.
  We determine $\nu(G_j[V \sm T])$ by
  considering the effects of removing $T\cap A_j$
  and subsequently $T\cap D_j$ from $G_j$.
  By \cref{lem:sigma}(\ref{it:sig:Y}), $A_j\subseteq T$, so by \cref{lem:stab},
  removing $T\cap A_j=A_j$ from $G_j$ reduces $\nu(G_j)$ by $|A_j|$.
  Writing $G_j^* = G_j[V \sm A_j]$,
  we have $\nu(G_j^*)=\nu(G_j)-|T\cap A_j|$, and $G_j^*$ is $D$-complete.
  Next we remove $T\cap D_j$ from $G_j^*$
  in steps, by removing $T\cap K_j^i$ for every $i\in[\iota_j]$.
  For each $i$, by \cref{lem:sigma}(\ref{it:sig:X}), either $T\supseteq K_j^i$,
  or $|T\cap K_j^i|$ is even.
  In the first case, by \cref{lem:K},
  removing $T\cap K_j^i$ decreases $\nu$ by $\lfloor|T\cap K_j^i|/2\rfloor$.
  In the second case, by \cref{lem:D:flation},
  removing $T\cap K_j^i$ decreases $\nu$ by $\lfloor|T\cap K_j^i|/2\rfloor$.
  Writing $G_j^\circ=G_j[V\sm T]$,
  we have $\nu(G_j^\circ)=\nu(G_j^*)-\sum_{i\in[\iota_j]}\lfloor|T\cap K_j^i|/2\rfloor$,
  and $G_j^\circ$ is $D$-complete.
  To conclude,
  \[
    \nu(G_j') = \nu(G_j)-\sum_{i\in[\iota_j]}\floor{|T\cap K_j^i|/2}-|T\cap A_j|.
  \]
  Since $T$ is $\sigma$-maximal we have $\sigma(T)\ge 0$,
  so (\ref{it:dec:mns}) follows from
  \[\begin{aligned}
    \mns(\cG')
    = \sum_j\nu(G_j') 
      &= \sum_j \nu(G_j)
         -\left(\sum_j\sum_{i\in[\iota_j]}\floor{|T\cap K_j^i|/2}
         +\sum_j|T\cap A_j|\right)\\
      &= \mns(\cG)-(\sigma(T)+|T|)
      \le \mns(\cG)-|T|.
  \end{aligned}\]

  Next, let $v\in\Theta(\cG)$. We will show that $v\notin T$.
  Indeed, by the definition of $\Theta$, and since $\cG$ is proper,
  for every $j\in[q]$, $v\notin A_j$,
  and for every $i\in[\iota_j]$, if $|K_j^i|>1$ then $v\notin K_j^i$.
  Thus, if $v\in T$,
  we would have $\sigma(T\sm\{v\})-\sigma(T)=1$,
  contradicting the $\sigma$-maximality of $T$.
  This settles~(\ref{it:dec:disjoint}).

  Suppose first that $G_0'$ is not complete.
  Since we haven't recoloured any uncoloured edges,
  we have $\Theta(\cG)\subseteq\Theta(\cG')$. 
  By the definition of $\unclr$, we also have $T\subseteq\Theta(\cG')$.
  On the other hand, if $v\notin \Theta(\cG)\cup T$
  then (by construction, properness of $\cG$, and (\ref{it:dec:disjoint})),
  $d_{G_0'}(v)=|\Theta(\cG)\cup T|$,
  hence, since $G_0'$ is not complete, $\Theta(\cG')=\Theta(\cG)\cup T$.
  Now, if $G_0'$ is complete, then, since 
  $d_{G_0'}(v)=|\Theta(\cG)\cup T|$ for every $v\notin \Theta(\cG)\cup T$,
  we have that $|\Theta(\cG)\cup T|\ge |V|-1$.
  In this case, we {\em set} $\Theta(\cG')=\Theta(\cG)\cup T$.
  This settles~(\ref{it:dec:theta}).

  Recall that $G_j'$ is obtained from $G_j^\circ$ by adding isolated vertices.
  Thus, since $G_j^\circ$ is $D$-complete, $G_j'$ is $D$-complete.
  Moreover, by \cref{lem:sigma}(\ref{it:sig:hf}), $\cK(\cG')$ is a hyperforest,
  hence $\cG'$ is $D$-acyclic.
  Also $\Theta(\cG')$ is a cover for $E(G_0')$,
  as every uncoloured edge of $\cG'$ is incident to either $\Theta(\cG)$ or $T$.
  Since we also kept the set of uncoloured edges disjoint from the set of coloured edges,
  we deduce that $\cG'$ is proper, settling (\ref{it:dec:prop}).
\end{proof}

\subsection{Peeling} \label{sec:peel}

We now come to the third stage of the main algorithm.
The decycling procedure of the previous subsection provides a $D$-acyclic colouring,
in which the cliques form a hyperforest, which we will now iteratively simplify by peeling
away the leaves one by one, using one of two methods: dissolution or merging. 
We start by discussing dissolution (see \cref{alg:dissolve,lem:dissol}).

\begin{algorithm}
  \caption{Clique dissolution}\label{alg:dissolve}
\begin{algorithmic}
  \Procedure{Dissolve}{$\cG=(G_0,\dots,G_q)$, $L$, $K$}
    \State {\bf assert} $\cG$ is $D$-acyclic
    \State {\bf assert} $L=K_j^i$ is a leaf of $\cK(\cG)$
    \State {\bf assert} $K=K_{j'}^{i'}$ is an edge of $\cK(\cG)$
    \State {\bf assert} $j'\ne j$
    \State {\bf assert $|K|\ge|L|$}
    \State $w\gets\link(L)$
    \State $\kappa\gets (|L|-1)/2$
    \State $K'\gets$ subset of $K$ of size $2\kappa+1$
    \State $S\gets$ subset of $L\sm \{w\}$ of size $\kappa$
    \State $G_{j}\gets\Call{D-isolate}{G_{j},L,S,K'}$
    \Comment \cref{alg:D:isol}
    \State $G_{j}[V\sm\Theta(\cG)]\gets\Call{CAD-complete}{G_{j}[V\sm\Theta(\cG)]}$
    \Comment \cref{alg:CAD:comp}
    \State $\cG,T\gets\Call{Decycle}{\cG}$
    \Comment \cref{alg:decycle}
    \State \Return $\cG,T$
  \EndProcedure
\end{algorithmic}
\end{algorithm}

\begin{lemma}[Clique dissolution]\label{lem:dissol}
  Let $\cG$ be a proper $D$-acyclic $q$-colouring of $G=(V,E)$,
  let $L$ be a leaf of $\cK(\cG)$,
  and let $K$ be an edge of $\cK(\cG)$ 
  of a different colour to $L$ with $|K|\ge|L|$.
  Write $\cG',T=\Call{Dissolve}{\cG,L,K}$.
  Then
  \begin{enumerate}
    \item\label{it:dis:prop} $\cG'$ is proper and $D$-acyclic;
    \item\label{it:dis:mns} $\mns(\cG')\le \mns(\cG)-|T|$;
    \item\label{it:dis:K} $|\cK(\cG')|<|\cK(\cG)|$;
    \item\label{it:dis:nu} $\nu(\cG')\le \nu(\cG)$;
    \item\label{it:dis:disjoint} $T\cap\Theta(\cG)=\es$;
    \item\label{it:dis:theta} $\Theta(\cG')=\Theta(\cG)\cup T$;
    \item\label{it:dis:ml} $m_\ell(\cG')\ge m_\ell(\cG)$.
  \end{enumerate}
\end{lemma}

For the proof of \cref{lem:dissol}, we will use the following binomial inequality.
\begin{lemma}\label{lem:binom:bd}
  For all integers $\kappa\ge 0$ and $\ell\ge 2$,
  \[
    \binom{3\kappa+1}{\ell} \ge 2\binom{2\kappa+1}{\ell}
  \]
\end{lemma}

\begin{proof}
  We will use the inequality
  \begin{equation}\label{eq:bin:1}
    \binom{a+b}{c}-\binom{a}{c} \ge b\binom{a}{c-1},
  \end{equation}
  for all integers $a,b\ge 0$ and $c\ge 1$.
  For a proof, let $A,B$ be disjoint sets of size $a,b$, respectively.
  Then the LHS counts the number of $c$-subsets of $A\cup B$ with at least one element in $B$,
  whereas the RHS counts the number of $c$-subsets of $A\cup B$ with exactly one element in $B$.
  We will also use the inequality
  \begin{equation}\label{eq:bin:2}
    x\binom{2x+1}{c-1} \ge \binom{2x+1}{c},
  \end{equation}
  for all integers $x\ge 0$ and $c\ge 2$.
  For a proof, 
  note that a simple double-counting argument gives
  \[
    (2x+1-(c-1))\binom{2x+1}{c-1}
    = c\binom{2x+1}{c},
  \]
  and the inequality follows since $(2x+1-(c-1))/c\le x$.
  Using \cref{eq:bin:1,eq:bin:2},
  we observe that
  \[
    \binom{3\kappa+1}{\ell}-\binom{2\kappa+1}{\ell}-\binom{2\kappa+1}{\ell}
    \ge \kappa\binom{2\kappa+1}{\ell-1}-\binom{2\kappa+1}{\ell}\ge 0,
  \]
  as required.
\end{proof}

\begin{proof}[Proof of \cref{lem:dissol}]
  Let $j$ be the colour of $L$ (so $L=K_j^i$ for some $i\in[\iota_j]$),
  and let $j'\ne j$ be the colour of $K$.
  Set $w=\link(L)$.
  Write $|L|=2\kappa+1$,
  let $K'$ be an arbitrary subset of $K$ of size $2\kappa+1$,
  and let $S$ be an arbitrary subset of $L'=L\sm\{w\}$ of size $\kappa$.
  As $\cK$ is a hyperforest, we have $K' \cap S \subseteq K \cap L' = \es$,
  and $K' \subseteq K$ is scattered in $G_j$ as $j' \ne j$.
  Write $G_j^0=\Call{D-isolate}{G_j,L,S,K'}$,
  and obtain $G_j^\circ$ by $CAD$-completing $G_j^0[V\sm\Theta(\cG)]$.
  Write $\cG^\circ=(G_0,\dots,G_{j-1},G_j^\circ,G_{j+1},\dots,G_q)$,
  $\cA^\circ=\cA(\cG^\circ)$, $\cK^\circ=\cK(\cG^\circ)$,
  and $\sigma^\circ=\sigma_{\cG^\circ}$.
  Then $\cG^\circ$ is proper and $\Theta$-complete,
  and $\Theta(\cG^\circ)=\Theta(\cG)$.
  As $\cG',T=\Call{Decycle}{\cG^\circ}$ and $\cG^\circ$ is proper,
  by \cref{cor:CAD:comp,lem:D:isol} we have
  $\nu(G_j^\circ)=\nu(G_j^0)=\nu(G_j)$ and  $\cK^\circ=\cK\sm\{L\}$.

  We now show that $T=S$.
  By \cref{lem:D:isol} we have $\GE(G_j^\circ)=(\es,S,V\sm S)$,
  so $\cA^\circ=\{S\}$,
  and  $T\supseteq S$ by \cref{lem:sigma}(\ref{it:sig:Y}).
  Let $V^*=V\sm S$ and $\cG^*=\cG^\circ[V^*]$,
  and write $\cA^*=\cA(\cG^*)$, $\cK^*=\cK(\cG^*)$,
  and $\sigma^*=\sigma_{\cG^*}$.
  Now, crucially, $\cA^*=\es$ and $\cK^*=\cK^\circ$:
  as $S$ is disjoint from every clique of $\cK^\circ$,
  this follows from \cref{lem:stab,lem:K}.
  Thus,
  $\cK^*$ is a hyperforest,
  hence $\cG^*$ is $D$-acyclic.
  By \cref{lem:acyclic:stable},
  $\es$ is the only $\sigma^*$-maximal set.
  Also, as $\cK^*=\cK^\circ$ and  $\cA^\circ=\{S\}$,
  for every $U\subseteq V^*$
  we have $\sigma^*(U)=\sigma^\circ(S\cup U)$.
  Thus, writing $T=S\cup U$ for $U\subseteq V^*$,
  we have $\sigma^*(U)=\sigma^\circ(T)\ge 0$,
  so $U=\es$, i.e.~$T=S$.

  In particular, $|T|=\kappa$.
  Thus, by \cref{lem:decycle}(\ref{it:dec:prop}),
  $\cG'$ is proper and $D$-acyclic
  (settling (\ref{it:dis:prop}))
  and by \cref{lem:decycle}(\ref{it:dec:mns}),
  $\mns(\cG')\le \mns(\cG)-\kappa$
  (settling (\ref{it:dis:mns})).
  Moreover,
  since $|\cK(\cG)|\ge 2$, we deduce that
  $\cK(\cG')=\cK(\cG)\sm\{L\}$,
  and, in particular,
  $|\cK(\cG')|<|\cK(\cG)|$, settling (\ref{it:dis:K}).
  Also, by \cref{lem:decycle}(\ref{it:dec:nu}),
  $\nu(\cG')\le\nu(\cG^\circ)=\nu(\cG)$,
  settling (\ref{it:dis:nu}).
  Finally, \cref{lem:decycle}(\ref{it:dec:disjoint}),(\ref{it:dec:theta})
  prove (\ref{it:dis:disjoint}),(\ref{it:dis:theta}).

  Next we bound $m_\ell$.
  To this end, we consider the construction
  of $\cG'$ from $\cG$ in two steps.
  Write $Y=\Theta(\cG)$ and $y=|Y|$.
  First, we remove all edges in $L$, to obtain $\cG^*$.
  We note that the number of copies of $K_\ell$ destroyed by this operation is
  \begin{equation}\label{eq:ml:step1}
    m_\ell(\cG)-m_\ell(\cG^*)
    = \binom{y+2\kappa+1}{\ell} - \binom{y}{\ell} - (2\kappa+1)\binom{y}{\ell-1}.
  \end{equation}
  Indeed,
  we destroyed any copy of $K_\ell$ in $\Theta(\cG)\cup L$
  (first term),
  excluding copies completely contained in $\Theta(\cG)$
  (second term)
  or having exactly one vertex in $L$
  (third term).
  The next step is to
  connect every vertex in $S$ to any other vertex in $V$.
  The number of newly created copies of $K_\ell$ is
  \begin{equation}\label{eq:ml:step2}
    m_\ell(\cG')-m_\ell(\cG^*)
    \ge \binom{y+3\kappa+1}{\ell} - \binom{y+2\kappa+1}{\ell} - \kappa\binom{y}{\ell-1}
    + \kappa\binom{y+\kappa}{\ell-1} - \kappa\binom{y}{\ell-1}.
  \end{equation}
  Indeed,
  every $\ell$-tuple in $Y\cup S\cup K$ is now a copy of $K_\ell$ (first term),
  but among those, every copy of $K_\ell$ in $Y\cup K$ already exists (second term),
  and so does every copy of $K_\ell$ in $Y\cup S$ that has exactly one vertex in $S$ (third term).
  In addition, every $\ell$-tuple on $Y\cup L'$
  that has exactly one vertex in $L'\sm S$
  is now a copy of $K_\ell$ (fourth term),
  but this includes copies that do not have any vertex in $S$,
  and hence are not new (fifth term).
  Combining \cref{eq:ml:step1,eq:ml:step2}, we have
  \begin{equation*}
    \begin{aligned}
    m_\ell(\cG')-m_\ell(\cG)
    &\ge \binom{y+3\kappa+1}{\ell} - \binom{y+2\kappa+1}{\ell} - 2\kappa\binom{y}{\ell-1}
    + \kappa\binom{y+\kappa}{\ell-1}\\
    &\phantom{{} \ge {}} - \left(
      \binom{y+2\kappa+1}{\ell} - \binom{y}{\ell} - (2\kappa+1)\binom{y}{\ell-1}
    \right)\\
    &\ge \Gamma_y(\ell) :=
     \binom{y+3\kappa+1}{\ell} - 2\binom{y+2\kappa+1}{\ell} + \binom{y+1}{\ell}
    + \kappa\binom{y+\kappa}{\ell-1}.
    \end{aligned}
  \end{equation*}
  We prove by induction on $y\ge 0$ that $\Gamma_y(\ell) \ge 0$ for every $\ell\ge 1$.

  \begin{description}
    \item[Base case: $\ell=1$.]
      In this case,
      \[
        \Gamma_y(1) = y+3\kappa+1 - 2(y+2\kappa+1) + 1 + \kappa + y
        = 0.
      \]
    \item[Base case: $y=0$ and $\ell\ge 2$.]
      In this case, by \cref{lem:binom:bd},
      \[
        \Gamma_0(\ell) = \binom{3\kappa+1}{\ell} - 2\binom{2\kappa+1}{\ell} + \kappa\binom{\kappa}{\ell-1}
        \ge \binom{3\kappa+1}{\ell} - 2\binom{2\kappa+1}{\ell}
        \ge 0.
      \]
    \item[Step: assume $\Gamma_y(\ell)\ge 0$ for all $\ell\ge 1$, and let $\ell\ge 2$.]
      In this case,
      \[
      \begin{aligned}
        \Gamma_{y+1}(\ell)
        &= \binom{y+1+3\kappa+1}{\ell} - 2\binom{y+1+2\kappa+1}{\ell} + \binom{y+1+1}{\ell}
           + \kappa\binom{y+1+\kappa}{\ell-1} \\
        &= \left(\binom{y+3\kappa+1}{\ell}+\binom{y+3\kappa+1}{\ell-1}\right)
        - 2\left(\binom{y+2\kappa+1}{\ell}+\binom{y+2\kappa+1}{\ell-1}\right)\\
        &\phantom{{} = {}} + \left(\binom{y+1}{\ell}+\binom{y+1}{\ell-1}\right)
           + \kappa\left(\binom{y+\kappa}{\ell-1}+\binom{y+\kappa}{\ell-2}\right) \\
        &= \Gamma_{y}(\ell)+\Gamma_{y}(\ell-1) \ge 0.
      \end{aligned}
      \]
  \end{description}
  This settles (\ref{it:dis:ml}).
\end{proof}

We have now prepared all the ingredients for the peeling algorithm (see \cref{alg:peel}).
As mentioned above, we repeatedly remove leaves, using dissolution
or merging, according to the following case distinction.
Consider $\cK=\cK(\cG)$ with $|\cK| \ge 2$, 
and let $L$ be a smallest leaf of $\cK$.
\begin{itemize}
\item If there exists an edge $K$ of $\cK$ of a different colour to $L$
with $|K|\ge |L|$ then we dissolve $L$ in relation to $K$.
\item
Otherwise, there exists an edge $K$ of the same colour as $L$
with $|K| \ge |L|$ so that we can merge $L$ into $K$;
indeed, any other leaf can play the role of $K$.
\end{itemize}

\begin{algorithm}
  \caption{Peeling}\label{alg:peel}
\begin{algorithmic}
  \Procedure{Peel}{$\cG=(G_0,\dots,G_q)$}
    \State {\bf assert} $\cG$ is $D$-acyclic
    \State $S\gets\es$
    \While{$|\cK(\cG)|\ge 2$}
      \State $L\gets K_j^i = $ leaf of minimum size of $\cK$
      \If{$\exists K=K_{j'}^{i'}$ with $j'\ne j$ and $|K|\ge|L|$}
        \State $\cG,T\gets$ \Call{Dissolve}{$\cG$, $L$, $K$}
        \Comment \cref{alg:dissolve}
        \State $S\gets S\cup T$
      \Else
        \State $K\gets K_j^{i'}$ with $i'\ne i$ and $|K|\ge|L|$
        \State $w\gets\link(L)$
        \State $G_j\gets$ \Call{Merge}{$G_j,L,w,K$}
        \Comment \cref{alg:D:merge}
      \EndIf
    \EndWhile
    \State \Return $\cG,S$
  \EndProcedure
\end{algorithmic}
\end{algorithm}

We showed above that the dissolution steps in the peeling algorithm
have the required properties; now we also do so for the merging steps.

\begin{lemma}[Merging in coloured graphs]\label{lem:merge:clr}
  Let $\cG$ be a proper $D$-acyclic $q$-colouring of $G=(V,E)$.
  Fix $j\in[q]$,
  and let $L,K$ be two distinct nontrivial cliques in $G_j$ with $|L|\le|K|$.
  Suppose $L$ is a leaf in $\cK(\cG)$,
  and let $w=\link(L)$.
  Set $G_j'=\Call{Merge}{G_j,L,w,K}$ and obtain $\cG'$ from $\cG$ by replacing $G_j$ with $G_j'$.
  Then
  \begin{enumerate}
    \item\label{it:mrg:nu} $\nu(\cG')=\nu(\cG)$;
    \item\label{it:mrg:mns} $\mns(\cG')=\mns(\cG)$;
    \item\label{it:mrg:theta} $\Theta(\cG')=\Theta(\cG)$;
    \item\label{it:mrg:prop} $\cG'$ is proper and $D$-acyclic;
    \item\label{it:mrg:K} $|\cK(\cG')|<|\cK(\cG)|$;
    \item\label{it:mrg:ml} $m_\ell(\cG')\ge m_\ell(\cG)$.
  \end{enumerate}
\end{lemma}

For the proof of \cref{lem:merge:clr}, we will use the following simple combinatorial inequality.
\begin{lemma}\label{lem:iep}
  For all non-negative integers $n,a,b,r$,
  one has
  \[ \binom{n+a+b}{r} \ge \binom{n+a}{r} + \binom{n+b}{r} - \binom{n}{r}. \]
\end{lemma}

\begin{proof}
  Let $X=Y\cup A\cup B$ be a partition with $|Y|=n$, $|A|=a$, and $|B|=b$.
  The LHS counts the number of $r$-subsets of $X$,
  whereas the RHS counts the number of $r$-subsets of $Y\cup A$ or $Y\cup B$.
\end{proof}

\begin{proof}[Proof of \cref{lem:merge:clr}]
  According to \cref{lem:D:merge}, $G_j'$ is $D$-complete,
  hence $\cG'$ is $AD$-complete.
  \Cref{lem:D:merge} also implies that $\nu(G_j')=\nu(G_j)$.
  Since no other colour has changed,
  we deduce that $\nu(\cG')=\nu(\cG)$ (settling (\ref{it:mrg:nu}))
  and $\mns(\cG')=\mns(\cG)$ (settling (\ref{it:mrg:mns})).
  Furthermore, the $\Call{D-merge}{}$ operation
  solely modifies edges within $G_j$ and does not alter the uncoloured graph $G_0$;
  consequently, $\Theta(\cG')=\Theta(\cG)$, settling (\ref{it:mrg:theta}).

  To see that $\cG'$ is proper, note that
  all edge modifications made by $\Call{D-merge}{}$ to form $G_j'$
  involve only vertices in $L\cup K$, which are in $V\sm\Theta(\cG)$ by properness of $\cG$. 
  To show that $\cG'$ is $D$-acyclic, it suffices to show that $\cK(\cG')$ is a hyperforest.
  But this holds since $I_{\cK(\cG')}$ is obtained from $I_{\cK(\cG)}$
  by deleting the vertex $L$ and connecting $K$ with every $u\in L':=L\sm\{w\}$.
  Since any such $u$ had $L$ as its only neighbour, we have not created any cycles in this process.
  This also shows that $|\cK(\cG')|=|\cK(\cG)|-1$,
  settling (\ref{it:mrg:prop}),(\ref{it:mrg:K}).

  It remains to show that the number of copies of $K_\ell$ cannot decrease.
  To this end,
  write $y=|\Theta(\cG)|$,
  $|L|=2\lambda+1$,
  and $|K|=2\kappa+1$,
  and note that
  \[
    \begin{aligned}
      m_\ell(\cG')-m_\ell(\cG)
      &\ge \binom{y+2\kappa+2\lambda+1}{\ell}
          - \binom{y+2\kappa+1}{\ell}
          - \binom{y+2\lambda+1}{\ell}
          + \binom{y}{\ell}
          + \binom{y}{\ell-1}\\
      &= \binom{y+2\kappa+2\lambda+1}{\ell}
         - \binom{y+2\kappa+1}{\ell}
         - \binom{y+2\lambda+1}{\ell}
         + \binom{y+1}{\ell}.
    \end{aligned}
  \]
  Indeed, every $\ell$-tuple in $Y\cup K\cup L'$ is now a copy of $K_\ell$ (first term),
  but among these every tuple in $Y\cup K$ already existed in $\cG$ (second term).
  We also removed every copy of $K_\ell$ that was contained in $Y\cup L$ (third term).
  We add the fourth term to account for double removal of copies in $Y$,
  and the fifth term to account for any copy
  that was contained in $Y\cup\{w\}$ that included $w$.
  We deduce from \cref{lem:iep} that this expression is nonnegative.
  This settles (\ref{it:mrg:ml}).
\end{proof}

\begin{figure*}[t!]
  \captionsetup{width=0.879\textwidth,font=small}
  \centering
  \begin{subfigure}[t]{0.29\textwidth}
    \centering
      \includegraphics[width=\textwidth]{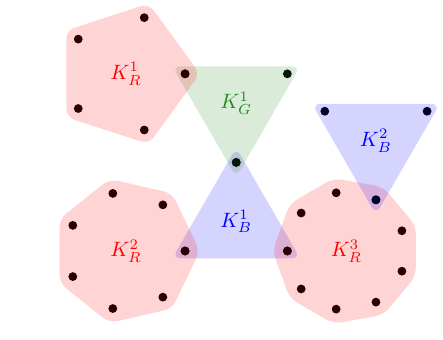}
    \caption{A $D$-acyclic coloured graph.}
    \label{fig:peel:1}
  \end{subfigure}%
  ~
  \begin{subfigure}[t]{0.04\textwidth}
    ~
  \end{subfigure}
  \begin{subfigure}[t]{0.29\textwidth}
    \centering
      \includegraphics[width=\textwidth]{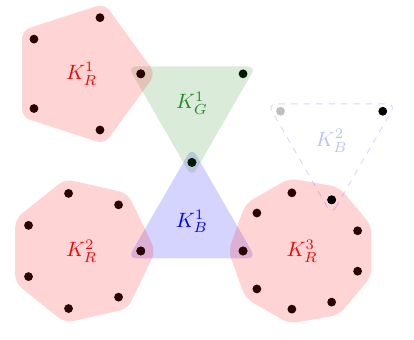}
    \caption{
      Dissolving $K_B^2$ relative to $K_R^3$.
    }
    \label{fig:peel:2}
  \end{subfigure}
  \begin{subfigure}[t]{0.04\textwidth}
  ~
  \end{subfigure}
  \begin{subfigure}[t]{0.29\textwidth}
    \centering
      \includegraphics[width=\textwidth]{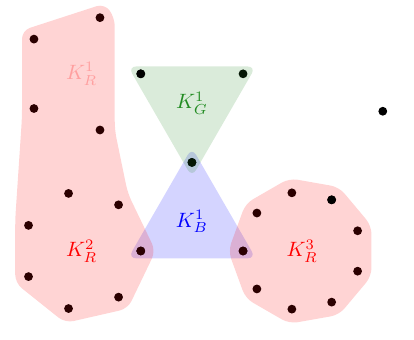}
    \caption{
      Merging $K_R^1$ into $K_R^2$.
    }
    \label{fig:peel:3}
  \end{subfigure}
  \begin{subfigure}[t]{0.29\textwidth}
    \centering
      \includegraphics[width=\textwidth]{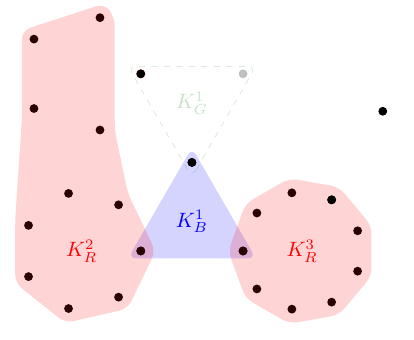}
    \caption{
      Dissolving $K_G^1$ relative to $K_B^1$.
    }
    \label{fig:peel:4}
  \end{subfigure}%
  ~
  \begin{subfigure}[t]{0.04\textwidth}
    ~
  \end{subfigure}
  \begin{subfigure}[t]{0.29\textwidth}
    \centering
      \includegraphics[width=\textwidth]{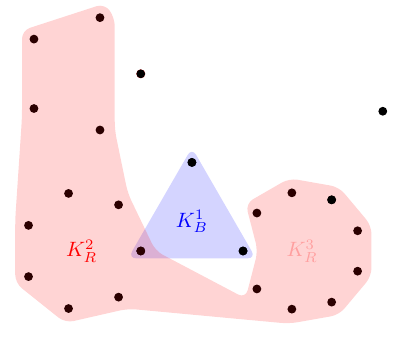}
    \caption{
      Merging $K_R^3$ into $K_R^2$.
    }
    \label{fig:peel:5}
  \end{subfigure}
  \begin{subfigure}[t]{0.04\textwidth}
  ~
  \end{subfigure}
  \begin{subfigure}[t]{0.29\textwidth}
    \centering
      \includegraphics[width=\textwidth]{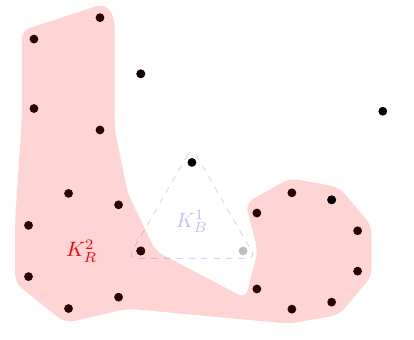}
    \caption{
      Dissolving $K_B^1$ relative to $K_R^2$.
    }
    \label{fig:peel:6}
  \end{subfigure}
  \caption{Peeling (\cref{alg:peel}).}
  \label{fig:peel}
\end{figure*}

We conclude this subsection with the analysis of the peeling algorithm.

\begin{lemma}[Peeling]\label{lem:peel}
  Let $\cG$ be a proper $D$-acyclic $q$-colouring of $G=(V,E)$.
  Then $\Call{Peel}{\cG}$ terminates.
  Letting $(\cG',S)=\Call{Peel}{\cG}$,
  we have 
  \begin{enumerate}
    \item\label{it:peel:prop} $\cG'$ is proper and $D$-acyclic;
    \item\label{it:peel:K} $|\cK(\cG')|\le 1$;
    \item\label{it:peel:ml} $m_\ell(\cG')\ge m_\ell(\cG)$;
    \item\label{it:peel:nu} $\nu(\cG')\le \nu(\cG)$;
    \item\label{it:peel:mns} $\mns(\cG')\le \mns(\cG)-|S|$;
    \item\label{it:peel:disjoint} $S\cap\Theta(\cG)=\es$;
    \item\label{it:peel:theta} $\Theta(\cG')=\Theta(\cG)\cup S$.
  \end{enumerate}
\end{lemma}

\begin{proof}
  The fact that $\Call{Peel}{\cG}$ terminates,
  with $\cG'$ being proper and $D$-acyclic and $|\cK(\cG')|\le 1$,
  follows from \cref{lem:dissol}(\ref{it:dis:K}),(\ref{it:dis:prop})
  and \cref{lem:merge:clr}(\ref{it:mrg:K}),(\ref{it:mrg:prop}).
  This settles (\ref{it:peel:prop}) and (\ref{it:peel:K}).
  We deduce (\ref{it:peel:ml}) from \cref{lem:dissol}(\ref{it:dis:ml}) and \cref{lem:merge:clr}(\ref{it:mrg:ml}),
  (\ref{it:peel:nu}) from \cref{lem:dissol}(\ref{it:dis:nu}) and \cref{lem:merge:clr}(\ref{it:mrg:nu}),
  and (\ref{it:peel:mns}) from \cref{lem:dissol}(\ref{it:dis:mns}) and \cref{lem:merge:clr}(\ref{it:mrg:mns}).
  Furthermore, $S$ is disjoint from $\Theta(G)$
  because $S$ is a union of sets,
  each of which is a set $T$ returned by a $\Call{Dissolve}{}$ call;
  by \cref{lem:dissol}(\ref{it:dis:disjoint}), each such set $T$ is disjoint
  from the $\Theta$ of its input coloured graph,
  which is a superset of $\Theta(\cG)$;
  this settles (\ref{it:peel:disjoint}).
  Finally,
  \cref{lem:dissol}(\ref{it:dis:theta}) and \cref{lem:merge:clr}(\ref{it:mrg:theta})
  prove (\ref{it:peel:theta}).
\end{proof}

\subsection{Distilling} \label{sec:distil}

Finally, we can state and analyse our main algorithm (see  \cref{alg:distil,lem:distil}),
which is a combination of $3$ subroutines implementing the $3$ stages
discussed in the proof outline:
\begin{enumerate}[nosep]
\item each colour is $AD$-completed separately,
\item the GE-decompositions of all colours are simplified by decycling,
\item the peeling algorithm eliminates all but at most one clique.
\end{enumerate}

\begin{algorithm}
  \caption{Distilling}\label{alg:distil}
\begin{algorithmic}
  \Procedure{Distil}{$\cG=(G_1,\dots,G_q)$}
    \For{$j\in[q]$}
      \State $G_j\gets\Call{AD-complete}{G_j}$
      \Comment \cref{alg:AD:comp}
    \EndFor
    \State $\cG,T\gets\Call{Decycle}{\cG}$
    \Comment \cref{alg:decycle}
    \State $\cG,S\gets\Call{Peel}{\cG}$
    \Comment \cref{alg:peel}
    \State \Return $\cG,T\cup S$
  \EndProcedure
\end{algorithmic}
\end{algorithm}

\begin{lemma}[Distilling]\label{lem:distil}
  Let $\cG$ be a $q$-colouring of an $n$-vertex graph $G$.
  Suppose that no single colour class of $\cG$ contains a perfect matching, 
  and that $\mns(\cG)<n$.
  Let $\cG',S=\Call{Distil}{\cG}$,
  and write $\cG'=(G'_0,G'_1,\dots,G'_q)$.
  Then there exists $\kappa \in \NN$ for which the following hold:
  \begin{enumerate}
    \item\label{it:dist:1} In $\cG'$, at most one colour (say, colour $1$) contains edges.
    \item\label{it:dist:K} The graph $G'_1$ is the disjoint union
    of a clique $K$ and a (possibly empty) set of isolated vertices,
   where $K$ is of size $2\kappa+1$ and is disjoint from $S$.
    \item\label{it:dist:ml} $m_\ell(\cG')\ge m_\ell(\cG)$.
    \item\label{it:dist:nu} $\nu(\cG)\ge \nu(\cG')=(\kappa,0,\dots,0)$.
    \item\label{it:dist:mns} $\mns(\cG)\ge \mns(\cG')+|S| = \kappa+|S|$.
    \item\label{it:dist:vc} $S$ is a cover for $E(G_0')$.
    \item\label{it:dist:cc} The underlying graph of $\cG'$ is a clique-cone graph
      with clique set $K$ and cone set $S$.
  \end{enumerate}
\end{lemma}

\begin{proof}
  The $\Call{Distil}{}$ procedure involves three main stages:
  AD-completion (of each colour graph),
  Decycling, and Peeling.
  Write $\cG=(G_1,\dots,G_q)$.
  \begin{description}
    \item[AD-completion] 
      For $j\in[q]$, write $G_j^\circ=\Call{AD-complete}{G_j}$,
      and let $\cG^{\circ}=(G_1^\circ,\dots,G_q^\circ)$.
      Recall that $\nu(G_j)<n/2$ for every $j\in[q]$.
      Thus, by \cref{cor:AD:comp}, we know that
      (a1) $G_j$ is a subgraph of $G_j^\circ$ for every $j\in[q]$,
      implying that $m_\ell(\cG^\circ)\ge m_\ell(\cG)$;
      (a2) $\nu(G_j^\circ)=\nu(G_j)$,
      implying that $\nu(\cG^\circ)=\nu(\cG)$ and $\mns(\cG^\circ)=\mns(\cG)$;
      and (a3) $\cG^\circ$ is $AD$-complete.

    \item[Decycling]
      By (a3),
      and since $\Theta(\cG^\circ)=\es$,
      the assertions that the input for \Call{Decycle}{} is $\Theta$-complete and proper hold.
      Let $\cG^*,T=\Call{Decycle}{\cG^\circ}$.
      By \cref{lem:decycle},
      (b1) $\cG^*$ is proper and $D$-acyclic;
      (b2) $m_\ell(\cG^*)\ge m_\ell(\cG^\circ)$;
      (b3) $\nu(\cG^*) \le \nu(\cG^\circ)$;
      (b4) $\mns(\cG^*) \le \mns(\cG^\circ)-|T|$;
      and (b5) $\Theta(\cG^*)= T$.
           
    \item[Peeling]
      By (b1), the assertion that the input for \Call{Peel}{} is proper and $D$-acyclic is met.
      We may write $\cG',S'=\Call{Peel}{\cG^*}$,
      noting that $S=T\cup S'$.
      By \cref{lem:peel},
      (c1) $\cG'$ is proper and $D$-acyclic;
      (c2) $|\cK(\cG')|\le 1$;
      (c3) $m_\ell(\cG')\ge m_\ell(\cG^*)$;
      (c4) $\nu(\cG')\le \nu(\cG^*)$;
      (c5) $\mns(\cG') \le \mns(\cG^*)-|S'|$;
      (c6) $S'\cap\Theta(\cG^*)=\es$;
      and (c7) $\Theta(\cG')=S'\cup \Theta(\cG^*)$.
  \end{description}

  Now, (c1) and (c2) imply that there is at most $1$ colour with edges,
  and in this colour there is at most $1$ connected component of size greater than $1$,
  which is an odd clique.
  Assume, without loss of generality, that $G_2',\dots,G_q'$ are all empty.
  This settles (\ref{it:dist:1}).
  Statements (a1), (b2), and (c3) imply (\ref{it:dist:ml}).
  By (a2), (b4), and (c5) we deduce $\mns(\cG')\le\mns(\cG)-|T|-|S'|$.
  By (b5) and (c6) we deduce $S'\cap T=\es$, hence
  $\mns(\cG')\le\mns(\cG)-|T\cup S'|=\mns(\cG)-|S|$,
  which implies (d1) $\mns(\cG)\ge\mns(\cG')+|S|$.
  This, combined with the assumption that $\mns(\cG)<n$,
  implies (d2) $S\subsetneq V$.
  Now, note that (b5) and (c7) imply that
  (d3) $\Theta(\cG')=S'\cup\Theta(\cG^*)=S'\cup T=S$.
  If $G_1'$ has a nontrivial clique, denote it by $K$,
  and note that since $\cG'$ is proper (by (c1)), $K$ is disjoint from $\Theta(\cG')$,
  and hence from $S$.
  Otherwise, since $S\subsetneq V$ (by (d2)),
  let $K=\{v\}$ for $v\notin S$, and again $K$ is disjoint from $S$.
  Write $|K|=2\kappa+1$ for an integer $\kappa\ge 0$;
  this settles (\ref{it:dist:K}).
  (\ref{it:dist:nu}) then
  follows from (a2), (b3), (c4), (\ref{it:dist:1}), and~(\ref{it:dist:K}),
  and~(\ref{it:dist:mns}) follows from (d1) and~(\ref{it:dist:nu}).
  Finally,
  let $G'$ be the underlying graph of $\cG'$.
  Partition the vertex set $V$ into $S\cup K\cup Z$.
  Evidently, $G'[K]$ is a clique and $S$ is (by (d3)) a cone-set in $G'$.
  (c1) and (d3) also imply that $S$ is a cover for $E(G'_0)$
  (settling~(\ref{it:dist:vc})),
  hence there are no uncoloured edges between $K,Z$;
  additionally, by (\ref{it:dist:1}), there are no coloured edges between $K,Z$.
  We deduce $E_{G'}(K,Z)=\es$, implying that $G'$ is a clique-cone graph
  with clique set $K$ and cone set $S$.
  This settles~(\ref{it:dist:cc}).
\end{proof}

\begin{proof}[Proof of \cref{thm:max}]
  We start by disposing of the cases $\ell \le n \le \tmax+\Lambda_\vt$.
  In the introduction, we described a colouring of the complete graph $K_{\tmax+\Lambda_\vt}$ that is $\vt K_2$-free.
  Restricting this colouring to the vertex set of $K_n$ shows that $K_n \nto \vt K_2$.
  Since $K_n$ contains the maximum possible number of $K_\ell$ copies,
  it is the extremal graph in this range.
  As $K_n$ is a clique-cone graph (e.g., $G_{n,1,n-1}$),
  the theorem holds.  We may therefore assume that $n > \tmax+\Lambda_\vt$.

  Let $G$ be a graph with $m_\ell(G)=\grt_\ell(n\to\vt K_2)$,
  and let $\cG=(G_1,\dots,G_q)$ be a $\vt K_2$-free $q$-colouring of $G$.
  In particular, $\mns(\cG)\le\Lambda_\vt$.
  Note that $n>\tmax+\Lambda_\vt$ implies $n>2(\tmax-1)$ and $n>\Lambda_\vt$.
  In particular, no $G_j$ admits a perfect matching and $n>\mns(\cG)$.
  
  Let $\cG',S=\Call{Distil}{\cG}$ and $G'$ be the underlying graph of $\cG'$.
  By \cref{lem:distil}(\ref{it:dist:cc}),
  $G'$ is a clique-cone graph $G_{n,2\kappa+1,|S|}$  for some $\kappa \in \NN$.
  By \cref{lem:distil}(\ref{it:dist:ml})  and  \cref{lem:distil}(\ref{it:dist:mns}) we have
    $m_\ell(G') \ge m_\ell(G)$ and $\kappa+|S|\le\mns(\cG)$.
  We now colour every uncoloured edge of $\cG'$, obtaining $\cG^+=(G_1^+,\dots,G_q^+)$, as follows.
  Let $s_1=t_1-1-\kappa$ and $s_j=t_j-1$ for $j=2,\dots,q$,
  so $\sum_j s_j = \Lambda_\vt-\kappa\ge \mns(\cG)-\kappa\ge |S|$.
  Let $S=S_1\cup\dots\cup S_q$ be a partition of $S$ (where parts may be empty)
  such that $|S_j|\le s_j$ for $j\in[q]$.
  Note that by \cref{lem:distil}(\ref{it:dist:vc}), $S$ is a cover for $E(G_0')$.
  Let $F_j$ be the set of edges incident to a vertex of $S_j$,
  and colour every edge of $F_j$ by colour $j$,
  thus colouring every uncoloured edge.
  Note that $m_\ell(\cG^+)=m_\ell(\cG')\ge m_\ell(\cG)$,
  and the underlying graph of $\cG^+$ is $G'$,
  which is a clique-cone graph $G_{n,x,y}$ with $x=2\kappa+1\in[1,2\tmax-1]$
  (since $\kappa\le \tmax-1$, by \cref{lem:distil}(\ref{it:dist:nu})).
  Finally,
  for every $j\in[q]$, by \cref{lem:cover},
  $\nu(G_j^+) \le \nu(G_j')+\tau(F_j)$.
  Since $\tau(F_j)\le s_j$, we have $\nu(\cG^+)\le\nu(\cG')+(s_1,\dots,s_q)=\vt-\one_q$.
\end{proof}

\section{Proof of \cref{thm:grt}}\label{sec:phi}
Now we will deduce \cref{thm:grt} from \cref{thm:max}.
Let $\vt\in\NN_+^q$ and $n\ge\max\{\ell,\tmax+\Lambda_\vt\}$ .
By \cref{thm:max}, the value of $\grt_\ell(n \to (t_1K_2,\dots,t_qK_2))$
is given by the maximum of $m_\ell(G_{n,x,y}) = \phi_{\ell,n}(x,y)$ over all $(x,y)$ in 
\[ \cA := \{ (x,y) \in \NN^2:\ x+y\le n,\ 1\le x\le 2\tmax-1,\
 G_{n,x,y}\nto (t_1K_2,\dots,t_qK_2). \}\]
To deduce \cref{thm:grt}, we will show that $\cA$ is contained in the following simpler
set $\cA_0$, and that the maximum value of $\phi:=\phi_{\ell,n}$ on $\cA_0$
is achieved at one of two points, which belong to $\cA$ and correspond
to the sparse and dense constructions discussed in the introduction.

\begin{claim}\label{cl:A}
Writing $\Lambda:=\Lambda_\vt$, we have
  \[
    \cA\subseteq\cA_0 := \{
      (x,y)\in \NN^2 :\  y\le \Lambda-\floor{x/2}
    \}.
  \]
\end{claim}
\begin{proof}
  First note that $\nu(G_{n,x,y})=y+\floor{x/2}$.
  Let $(x,y)\in\cA$.
  By definition, there exists a $\vt K_2$-free colouring $\cG=(G_1,\dots,G_q)$ of $G_{n,x,y}$.
  In particular, $\nu(G_j)\le t_j-1$ for all $j\in[q]$.
  Thus, $y+\floor{x/2}=\nu(G_{n,x,y})\le \sum_j \nu(G_j) \le \Lambda$,
  implying $y\le\Lambda-\floor{x/2}$, hence $(x,y)\in\cA_0$.
\end{proof}

Next we observe that $\phi_{\ell,n}(x,y)=m_\ell(G_{n,x,y})$ is monotone in both $x$ and $y$,
as increasing either corresponds to adding edges to the underlying graph.
Thus, the maximum of $\phi$ over $\cA_0$ is attained on its upper boundary,
defined by $y(x)=\Lambda-\floor{x/2}$. Moreover, the maximum has $x$ odd, 
as if $x$ is even then $\phi_{\ell,n}(x,y(x))<\phi_{\ell,n}(x+1,y(x))=\phi_{\ell,n}(x+1,y(x+1))$.
We thus consider
\[
  g(\kappa)
    =\phi(2\kappa+1,y(2\kappa+1))
    =\phi(2\kappa+1,\Lambda-\kappa), \text{ where } \kappa=\floor{(x-1)/2}.
\]

\begin{claim}\label{cl:g:convex}
  $g(\kappa)$ is convex in $\NN\cap[0,\tmax-1]$.
\end{claim}

\begin{proof}
  For $0\le\kappa\le \tmax-2$,
  let $\Delta g(\kappa)=g(\kappa+1)-g(\kappa)$.
  Then
  \begin{equation*}
    \begin{aligned}
      \Delta g(\kappa)
      &= \phi(2\kappa+3,\Lambda-\kappa-1)
        - \phi(2\kappa+1,\Lambda-\kappa)\\
      &= \binom{\kappa+2+\Lambda}{\ell}+\binom{\Lambda-\kappa-1}{\ell-1}(n-\kappa-2-\Lambda)\\
      &\phantom{{}={}} -
         \binom{\kappa+1+\Lambda}{\ell}-\binom{\Lambda-\kappa}{\ell-1}(n-\kappa-1-\Lambda)\\
      &= \left[\binom{\kappa+2+\Lambda}{\ell}-\binom{\kappa+1+\Lambda}{\ell}\right]\\
      &\phantom{{}={}} -
        \left[
         \binom{\Lambda-\kappa}{\ell-1}(n-\kappa-1-\Lambda)
         -\binom{\Lambda-\kappa-1}{\ell-1}(n-\kappa-2-\Lambda)
         \right]\\
      &= \overbrace{\binom{\kappa+1+\Lambda}{\ell-1}}^A
        -\left[
         \overbrace{\binom{\Lambda-\kappa-1}{\ell-2}}^B
         \overbrace{(n-\kappa-1-\Lambda)}^C
         +\overbrace{\binom{\Lambda-\kappa-1}{\ell-1}}^D
         \right].
    \end{aligned}
  \end{equation*}
  As $n\ge \tmax+\Lambda$ and $\kappa\le \tmax-2$,
  we have $C=n-\kappa-1-\Lambda>0$.
  Now, since $\binom{x}{\ell}$ is monotone increasing in $x$,
  we deduce that $A$ is increasing (in $\kappa$),
  and that $B$, $C$, and $D$ are decreasing (and nonnegative).
  Therefore, $\Delta g(\kappa)=A-BC-D$ is increasing, hence $g(\kappa)$ is convex.
\end{proof}

\begin{proof}[Proof of \cref{thm:grt}]
  As $\cA\subseteq\cA_0$  by \cref{cl:A}, we have 
  $\max_{(x,y)\in\cA}\phi(x,y)\le\max_{(x,y)\in\cA_0}\phi(x,y)$.
  By \Cref{cl:g:convex} and the preceding discussion we have
  $\max_{(x,y)\in\cA_0} \phi(x,y) =
   \max_{\kappa\in\{0,\dots,\tmax-1\}}g(\kappa)=\max\{g(0),g(\tmax-1)\}$.
  Also, $\cA$ contains $(1,\Lambda)$ and $(2\tmax-1,\Lambda-\tmax+1)$,
  as these points correspond to the sparse and the dense constructions
  discussed in the introductions.  We deduce that
  $\max_{(x,y)\in\cA}\phi(x,y)=\max\{g(0),g(\tmax-1)\}$,
  which proves \cref{thm:grt}.
\end{proof}

\section{Concluding remarks}
\paragraph{The inverse problem}
We can reformulate the diagonal case (all $t_i$ equal) 
of our result as the following inverse problem.
Given a graph $G$ with $m$ copies of $K_\ell$,
determine the minimum possible value of $\nu_q(G)$,
defined as the largest integer $k$ such that $G \to_q kK_2$.
While our results give an implicit characterisation of the solution,
it seems difficult to give an explicit formula for general $\ell$.
However, when $\ell=2$, we can give an explicit asymptotic solution, as follows.

For an integer $q\ge 1$ and edge density $\alpha\in[0,1]$, define\footnote{
The limit  $x\to q$ is only needed when $q=3$ due to the removable singularity.
}
\[
  \mathfrak{s}_q(\alpha)=\frac{1-\sqrt{1-\alpha}}{q},\qquad
  \mathfrak{d}_q(\alpha) =
  \lim_{x\to q}\frac{x-1-\sqrt{1-2x+x^2-\alpha(x^2-2x-3)}}{x^2-2x-3}\ ,
\]
and let $M(q)=4(q^2+3q)/(2q+3)^2$.
Here, $\mathfrak{s}_q$ corresponds to the (asymptotic, normalised) size
of the guaranteed monochromatic matching in the sparse construction,
$\mathfrak{d}_q$ corresponds to the size
of the guaranteed monochromatic matching in the dense construction,
and $M(q)$ is the value of $\alpha$ for which
$\mathfrak{s}_q(\alpha)=\mathfrak{d}_q(\alpha)$. The
expression $1-\sqrt{1-\alpha}$ in the definition of $\mathfrak{s}_q(\alpha)$
is the guaranteed size of a matching in the sparse construction;
so $\mathfrak{s}_q$ is obtained from it by the pigeonhole principle.

\begin{theorem}\label{thm:l=2:asymp}
  Let $q \in \NN_+$ and $G$ be a graph on $n$ vertices
  with $\alpha\binom{n}{2}$ edges. Then
  \[
    \nu_q(G) \ge \min\{\mathfrak{s}_q(\alpha),\mathfrak{d}_q(\alpha)\} \cdot n - o(n).
  \]
\end{theorem}

\begin{figure}[t!]
\captionsetup{width=0.879\textwidth,font=small}
  \centering
  %----- Subfigure for q=1 -----
  \begin{subfigure}[t]{0.48\textwidth}
    \centering
    \hspace*{\fill}%
    \includegraphics{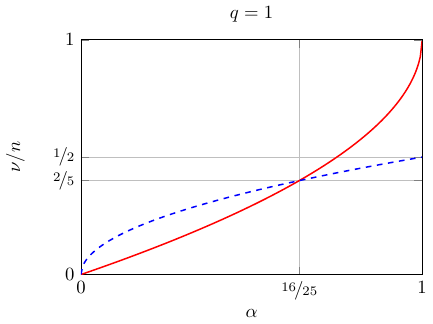}
  \end{subfigure}
  %----- Subfigure for q=2 -----
  \begin{subfigure}[t]{0.48\textwidth}
    \centering
    \hspace*{\fill}%
    \includegraphics{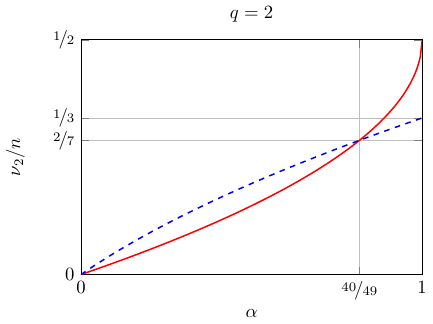}
  \end{subfigure}
  \hfill\\
  %----- Subfigure for q=3 -----
  \begin{subfigure}[t]{0.48\textwidth}
    \centering
    \hspace*{\fill}%
    \includegraphics{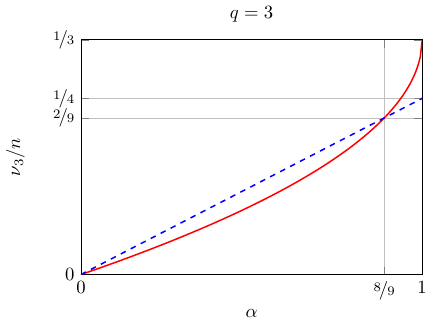}
  \end{subfigure}
  %----- Subfigure for q -----
  \begin{subfigure}[t]{0.48\textwidth}
    \centering
    \hspace*{\fill}%
    \includegraphics{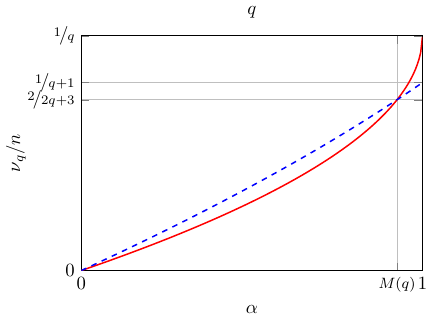}
  \end{subfigure}
  
  \caption{A visualisation of the asymptotics of $\nu_q/n$
  in terms of $\alpha=m_2/\binom{n}{2}$ for
  $q=1$ (the Erd\H{o}s--Gallai edge bound),
  $q=2$, $q=3$, and general $q$ (schematic).
  The solid, red line corresponds to the sparse regime ($\alpha\le M(q)$)---%
  this is the function $\mathfrak{s}_q$,
  while the dashed, blue line corresponds to the dense regime ($\alpha\ge M(q)$)---%
  this is the function $\mathfrak{d}_q$.
  Note that $\mathfrak{s}_q$ is convex,
  while $\mathfrak{d}_q$ is concave for $q\le 2$,
  linear for $q=3$,
  and convex for $q\ge 4$.}
  \label{fig:l2}
\end{figure}

Thus the guaranteed matching size is determined by the lower envelope of two curves;
see \cref{fig:l2}.

\paragraph{Counting other graphs}
A natural direction for future research on generalised Ramsey--Tur\'an problems for matchings
is to consider other enumerated graphs $T$ besides cliques.
We expect that the structural methodology of our proof would extend to this setting,
although our clique-counting arguments are quite delicate in places,
so it may be a challenge to extend these to general graphs.

\paragraph{Hypergraphs}
For the multicolour Ramsey number of matchings,
Alon, Frankl, and Lov\'asz~\cite{AFL86}
generalised the Cockayne--Lorimer Theorem to uniform hypergraphs.
In contrast, the corresponding Tur\'an problem for hypergraphs
(the famous Erdős Matching Conjecture)
is known for large $n$ but
still unresolved in general (see~\cite{FK22} for the current state of the art).
Similarly, generalised Tur\'an results for hypergraph matchings
are known~\cite{LW20} for large $n$ but open in general.
We are not aware of any work on (generalised) 
Ramsey--Tur\'an problems  for hypergraph matchings,
so this is another natural target for further research.
Our proof methodology relies fundamentally on the Gallai--Edmonds decomposition,
for which there is no known analogue for hypergraphs,
so we expect that new ideas and techniques will be required.

\bibliography{library}

\end{document}